\documentclass[12pt]{amsart}
\usepackage[utf8]{inputenc}

\usepackage{graphicx}

\usepackage{amsmath}
\usepackage{amsthm}
\usepackage{amssymb}
\usepackage{amsfonts}

\usepackage{tikz}
\usetikzlibrary{decorations.markings}
\usepackage{ifthen}

\newtheorem{theorem}{Theorem}[section]
\newtheorem{proposition}[theorem]{Proposition}
\newtheorem{lemma}[theorem]{Lemma}
\newtheorem{corollary}[theorem]{Corollary}

\theoremstyle{definition}
\newtheorem{definition}[theorem]{Definition}

\newtheorem{remark}[theorem]{Remark}

\DeclareMathOperator{\Star}{Star}
\DeclareMathOperator{\Aut}{Aut}
\DeclareMathOperator{\Epi}{Epi}
\DeclareMathOperator{\Stab}{Stab}
\DeclareMathOperator{\Ker}{Ker}
\DeclareMathOperator{\Sch}{Sch}
\DeclareMathOperator{\Cay}{Cay}
\DeclareMathOperator{\diam}{diam}

\DeclareMathOperator{\Cof}{Cof}
\DeclareMathOperator{\Comp}{Comp}

\newcommand{\N}{\mathbb{N}}
\newcommand{\Z}{\mathbb{Z}}
\newcommand{\Zd}{\Z/d\Z}

\newcommand{\B}{\mathcal{B}}

\title{SCHREIER GRAPHS OF SPINAL GROUPS}
\author{Tatiana Nagnibeda, Aitor Pérez}
\date{}

\begin{document}

\begin{abstract}
We study Schreier dynamical systems associated with a vast family of groups that hosts many known examples of groups of intermediate growth. We are interested in the orbital graphs for the actions of these groups on $d-$regular rooted trees and on their boundaries, viewed as topological spaces or as spaces with measure. They form interesting families of finitely ramified graphs, and we study their combinatorics, their isomorphism classes and their geometric properties, such as growth and the number of ends.
\end{abstract}

\maketitle

\tikzset{midarrow>/.style={decoration={markings, mark=at position 0.5 with {\arrow{>}}}, postaction={decorate}}}

\tikzset{
	pics/Loop/.style n args={4}{
		code = {
			\begin{scope}[shift={#1}, rotate=#2]
				\draw[color=#4] (0,0) .. controls (#3/2, #3) and (-#3/2, #3) .. (0,0);
				\draw[fill=black] (0,0) circle (#3/32);
			\end{scope}
		}	
	}
}

\tikzset{
	pics/Line/.style n args={4}{
		code = {
			\begin{scope}[shift={#1}, rotate=#2]
				\draw[color=#4] (0,0) -- (0: #3);
			\end{scope}
		}	
	}
}

\tikzset{
	pics/Triangle/.style n args={4}{
		code = {
			\begin{scope}[shift={#1}, rotate=#2]
				\draw[color=#4] (0,0) -- (0: #3) -- (60: #3) -- cycle;
				\draw[fill=black] (0,0) circle (#3/32);
				\draw[fill=black] (0: #3) circle (#3/32);
				\draw[fill=black] (60: #3) circle (#3/32);
			\end{scope}
		}	
	}
}

\tikzset{
	pics/Pentagon/.style n args={4}{
		code = {
			\begin{scope}[shift={#1}, rotate=#2]
				\draw[color=#4] (0,0) -- (#3, 0) -- (1.2*#3, -0.9*#3) -- (0.5*#3, -1.5*#3) -- (-0.2*#3, -0.9*#3) -- cycle;
				\draw[fill=black] (0,0) circle (#3/32);
				\draw[fill=black] (0: #3) circle (#3/32);
				\draw[fill=black] (1.2*#3, -0.9*#3) circle (#3/32);
				\draw[fill=black] (0.5*#3, -1.5*#3) circle (#3/32);
				\draw[fill=black] (-0.2*#3, -0.9*#3) circle (#3/32);
			\end{scope}
		}	
	}
}

\tikzset{
	pics/Gamma/.style ={
		code = {
			\draw (0,0) .. controls (\r, -\r) and (2*\r, -\r) .. (3*\r,0);
			\draw (0,0) .. controls (\r, \r) and (2*\r, \r) .. (3*\r,0);
			\draw (1.5*\r,0) node {\tiny $\Gamma_{#1}$};
		}
	}
}

\section{Introduction}
\label{sec:introduction}

To a finitely generated group $G$ and a finite generating set $S$ one naturally associates the Cayley graph $\Cay(G,S)$. Cayley graphs of infinite finitely generated groups have long been employed as powerful tools in geometric group theory and also studied in their own right as interesting examples of vertex-transitive graphs. One can think about Cayley graphs as representing the action of $G$ on itself by left multiplication. More recently, similar graphical representations of other actions of discrete groups, on a topological space, or on a space with measure, have gained attention. Given a dynamical system $(G, \mathcal X)$ where $\mathcal X$ is a topological space and $G$ acts on $\mathcal X$ by homeomorphisms, or $(G,\mathcal X,\mu)$ where $(\mathcal X,\mu)$ is a space with measure and $G$ acts by measure-preserving transformations, and a finite generating set $S$ in $G$, one considers the family $(\Gamma_x)_{x\in \mathcal X}$ of Schreier graphs  associated with the orbits of the action. More precisely, each orbit $G\cdot x$ gives rise to a graph $\Gamma_x$ with vertices the points of the orbit and edges of the form $(g\cdot x, sg \cdot x)$, with $g\in G, s\in S$.  By construction, Schreier graphs are directed, edge-labelled, rooted graphs, but depending on the context it may be useful to forget about labeling, orientation and, sometimes the rooting. They are regular of degree $|S|$. If the action is (essentially) free, then (almost) all Schreier graphs $(\Gamma_x)_{x\in \mathcal X}$ are isomorphic to the Cayley graph, so it is particularly interesting to consider Schreier graphs for non-free actions. If the action is minimal or ergodic, then the different orbital graphs $\Gamma_x$ are generically locally isomorphic, but don't have to be isomorphic in general.

The family of graphs $(\Gamma_x)_{x\in \mathcal X}$ lies naturally in the space $\mathcal G_{*,S}$ of rooted edge-labeled graphs endowed with the pointed Gromov-Hausdorff topology. The map $\Sch: \mathcal X \rightarrow \mathcal G_{*,S}$ that maps $x$ to the rooted graph $\Gamma_x$ allows to push the dynamics from the space $\mathcal X$ to the space of graphs, and gives rise to the so-called {\it Schreier dynamical system} in the space of rooted graphs. The group acts on its Schreier graphs by changing the root. A closely related object is the space of subgroups $Sub(G)$ of the group equipped with Tychonoff topology and with the action of $G$ by conjugation, together with the map $\Stab: \mathcal X \rightarrow Sub(G)$ that maps $x$ to the stabilizer subgroup $\Stab_G(x)$. It is easy to check that stabilizer subgroups are in one-to-one correspondence with the rooted directed and labeled Schreier graphs.

An important class of groups posessing non-trivial non-free minimal ergodic actions are groups of automorphisms of rooted spherically homogeneous trees. This class of groups came into prominence starting with the first discovery of groups of intermediate growth: the first example found by Grigorchuk \cite{Gri84} is a group of automorphisms of the rooted binary tree. An action by autormorphisms on a rooted spherically homogeneous tree $T$ that is transitive on the levels of the tree extends to a minimal action by homeomorphisms on the boundary of the tree $\partial T$ that is ergodic with respect to the Bernoulli measure. Here the boundary of the tree is as usual considered to be the space of infinite rays in the tree issued from the root. Dynamical properties of such boundary actions were first investigated in \cite{GNS00} and first examples of their Schreier graphs were studied in detail in \cite{BG00}. 

Since then  Schreier graphs have occupied an important place in the study of self-similar groups and beyond. Schreier dynamical systems were described for a number of important self-similar examples: the first Grigorchuk group \cite{Vor12}, the Basilica group \cite{DDMN10}, the Hanoi Towers group \cite{DGKPR}, groups generated by bounded automata \cite{BDN16} as well as for Thompson's group $F$ acting by homeomorphisms on the interval $(0,1)$ \cite{Sav15}. We refer to the survey \cite{Gri11} for more information on different applications and directions of research on Schreier dynamical systems. Here we would like to mention only that Schreier graphs of these and other groups found applications in the study of amenability \cite{JM13, JMMS18, Nek18}, random walks and other probabilistic models \cite{AAV13, DDN11, DDN12, JZ18}, automata \cite{Bon07}, Laplacian spectra \cite{GLN18, GLNS19, GNP19, BGJRT19}, subgroup structure of branch groups \cite{FG18}, unitary representations \cite{DG17}, IRS \cite{AGV14, HY18} and URS \cite{Fra18}, and this list is far from being complete.

In this paper we are interested in the class of spinal groups, introduced in~\cite{BS01} as a generalization of the first Grigorchuk group. It consists of countably many uncountable families $\mathcal M_{d,m}$ of finitely generated groups of automorphisms of regular rooted trees, where $d\geq 2$ denotes the degree of the rooted regular tree on which the groups act, and $m\geq 1$ is an additional integer parameter. Each $\mathcal M_{d,m}$ is what is called a self-similar family of groups (see Section 2 below for the precise definition of the groups). 
The class of spinal groups contains many exotic examples such as infinite torsion groups or groups of intermediate growth. The purpose of this paper is to describe the Schreier graphs $\Gamma_n, n\geq 0$ of spinal groups with respect to the natural spinal generating set for the action on the levels of the tree (finite Schreier graphs), and the Schreier graphs $\Gamma_\xi, \xi\in \partial T_d$  for the action on the boundary of the tree.
 
The paper is organized as follows. Spinal groups are defined in Section 2. In Section 3 we study their Schreier graphs for the action on the levels of the tree and show that they can be constructed recursively. A precise description is given in  
Proposition~\ref{prop:description_Gamma_n}. Infinite Schreier graphs $(\Gamma_\xi)_{\xi\in \partial T}$ are studied in Section 4, see in particular Proposition\ref{prop:description_Gamma_xi}. In Section \ref{sec:space_schreier_graphs} we study the map $\Sch: \partial T \to \mathcal{G}_{*, S}$, which assigns to each $\xi \in \partial T$ its rooted Schreier graph $(\Gamma_\xi, \xi)$. The closure of $\Sch(\partial T)$ in the space of rooted labeled graphs is described in Theorem~\ref{thm:description_space_graphs}.  Section~\ref{sec:number_of_ends}, and more precisely, Theorem~\ref{thm:number_ends} gives a characterization of which boundary points have one or two-ended graphs. In Theorem~\ref{thm:measures_ends} we compute the measure of each of these sets, and see that a typical Schreier graph is two-ended for a spinal group acting on the binary tree, and one-ended for a spinal group acting on a tree of degree $d\geq 3$. In Section~\ref{sec:isomorphisms} we investigate isomorphism classes of the Schreier graphs. As labeled graphs, they are never isomorphic. However, as unlabeled graphs, there are some nontrivial isomorphisms among them. 
In particular, Theorem~\ref{thm:rooted_isomorphism} shows that, for $d \ge 3$, two boundary points have isomorphic graphs iff they satisfy the compatibility condition given in Definition~\ref{def:compatibility}, namely, if they have zeros at the same positions and the words in between have suffixes $(d-1)^r$ with the same $r$. Moreover, the isomorphism classes have measure zero (Theorem~\ref{thm:measure_isomorphism_classes}). This gives us the Benjamini-Schramm limit of the sequence of graphs $\Gamma_n$. Finally, we illustrate our results with some examples in Section~\ref{sec:examples}.

\section{Spinal groups}
\label{sec:spinal_groups}

Let $d \ge 2$ be an integer and let $T_d$ be the $d$-ary infinite rooted tree. If we consider the alphabet $X = \{0, \dots, d-1\}$, vertices in $T_d$ are in bijection with the set $X^*$ of fininte words in $X$: the root is represented by the empty word $\emptyset$ and, if $v$ represents a vertex, $vi$ represents its $i$-th child, for $i \in X$. For $n \ge 0$, $X^n$ denotes the $n$-th level of $T_d$, whose vertices are at distance $n$ of the root and hence are represented by words in $X$ of length $n$. Any automorphism of $T_d$ must fix the root and hence maps $X^n$ to itself, for every $n \ge 0$. Any automorphism $g \in \Aut(T_d)$ can be described inductively by a permutation $\tau\in Sym(X)$ it induces on the vertices of the first level and its projections 
$g_i$, $i = 0, \dots, d-1$, to the $d$ subtrees attached at the root. Symbolically it can be written as $$g = \tau (g_0,...,g_{d-1}).$$
A group $G$ acting on a rooted tree $T_d$ is self-similar if for every $g\in G$ and for every $i \in X$, its projections $g|_i$ belong to $G$. 

Consider the automorphism $a \in \Aut(T_d)$ defined by $a(v_0 \dots v_n) = (v_0 + 1) v_1 \dots v_n$, for any $v_0 \dots v_n \in X^*$, where the sum is taken modulo $d$. This automorphism $a$ cyclically permutes the subtrees at the root, so that the corresponding $\tau=(01 \dots d-1)$ and all $g_i$'s are trivial automorphisms. We set $A = \langle a \rangle \le \Aut(T_d)$, and we have $A \cong \Zd$.

Let now $m \ge 1$ and $B = (\Zd)^m$, and let $\omega = (\omega_n)_n \in \Epi(B, A)^\N$ be a sequence of epimorphisms from $B$ to $A$. For every $b \in B$, we define the automorphism $b_\omega$ as follows

\footnotesize
\[
b_\omega(v_0\dots v_n) = \left\{ \begin{array}{cc}
(d-1)^r0\omega_{r}(b)(v_{r+1})v_{r+2}\dots v_n & \quad \text{if } v_0\dots v_r = (d-1)^r0\\
v & \text{otherwise}
\end{array}\right..
\]
\normalsize
These automorphisms are often called \emph{spinal} automorphisms, as they fix the rightmost ray of the tree (the \emph{spine}). They act as $a^j$ on the subtree rooted at $(d-1)^r0$, if $\omega_r(b) = a^j$.

If we let $\Omega = \Omega_{d,m} \subset \Epi(B, A)^\N$ be the set of sequences satisfying the condition
\[
\forall i \ge 0, \quad \bigcap\limits_{j\ge i} \Ker(\omega_j) = 1,
\]
then, for every $\omega \in \Omega$, the action of $B_\omega = \langle b_\omega \mid b \in B \rangle$ on $T_d$ is faithful. The set $\Omega$ is preserved by the shift $\sigma: \Omega \to \Omega$ defined by deleting the first symbol in the sequence.

Notice that all $b_\omega \in B_\omega$ stabilize all vertices of the first level, and hence decompose as
\[
b_\omega = (\omega_0(b_\omega), 1, \dots, 1, b_{\sigma\omega})
\]

Hence, for any $\omega \in \Omega$, we can consider the group defined as $G_\omega = \langle A, B_\omega \rangle$. We will abuse notation and write $B = B_\omega$. Such groups were first considered in~\cite{BS01} under the name of \emph{spinal groups}. In order to be precise, the definition in \cite{BS01} and in the subsequent \cite{BGS03} allows for more general $A$'s and $B$'s, however in the same time always assuming an additional condition on the sequence $\omega$, namely $\bigcup\limits_{j \ge i} \Ker(\omega_j) = B$ for all $i\ge 0$. This condition guarantees that the group $G_\omega$ is torsion. In this paper, we will use the term \emph{spinal groups} for the groups $G_\omega$ with $A=\Zd$, $B=(\Zd)^m$ and $\omega\in \Omega$ as defined above, but we will not assume the additional kernel condition and will consider both torsion and non-torsion groups. We will denote by $\mathcal M_{d,m}$ the family of groups $(G_\omega)_{\omega\in\Omega_{d,m}}$, for all $d\geq 2$, $m\geq 1$. Each spinal group $G_\omega$ comes naturally equipped with the spinal generating set $S_\omega = A \cup B_\omega \setminus \{1\}$.

Let us mention some particular examples. For and $d\geq 2$ and $m=1$ there is only one sequence in $\Omega_{d,m}$, the constant one. For $d=2$ the corresponding spinal group is the infinite dihedral group. For $d=3$ it is the so-called Fabrykowski-Gupta group, one of the early examples of groups of intermediate growth \cite{FG85, BP09}. The family $\mathcal M_{2,2}$ is the uncountable family of groups of intermediate growth constructed by Grigorchuk in \cite{Gri84}. 
In particular the first Grigorchuk group corresponds to the periodic sequence $(\pi_{01}\pi_{10}\pi_{11})^\infty$, where $\pi_{01},\pi_{10},\pi_{11}$ are the three non-trivial epimorphisms from $\Z_2^2$ to $\Z_2$.  It was shown in \cite{BS01} that all torsion spinal groups are of intermediate growth. It is open for non-torsion ones, except for the cases $d=2, m\geq 2$ (Grigorchuk's proof for $d=2,m=2$ \cite{Gri84} can be generalized for bigger $m$) and, recently, $d=3$ \cite{Fra19}.

\section{Schreier graphs on $X^n$}
\label{sec:schreier_finite}

Let $G$ be a group generated by a finite set $S$. If $G$ acts transitively on a set $Y$, its orbital Schreier graph $\Sch(G, Y, S)$ is the graph with vertex set $Y$ and with directed edges $(y, sy)$ for every $s \in S$ and $y \in Y$, labeled by $s$. Equivalently, for any subgroup $H \le G$, we define its associated Schreier graph $\Sch(G, H, S)$ as the graph with vertices elements of $G/H$ and directed edges $(gH, sgH)$ for every $s \in S$ and $gH \in G/H$, labeled by $s$. The orbital Schreier graph $\Sch(G, Y, S)$ is isomorphic to $\Sch(G, \Stab_G(y), S)$ for every $y \in Y$.

Accordingly, any finitely generated group $G\leq Aut(T_d)$ with a transivite action on every level $X^n$ gives rise to the graphs $\Gamma_n = \Sch(G, X^n, S)$, for $n \ge 0$. 

\begin{definition}
	Let $S$ be a finite set and $X = \{1, \dots, d\}$, and let $n \ge 0$. Let $\Gamma$ be a graph with edge labels in $S$ and vertex labels in $X^n$, and let $\Lambda$ be a finite graph on $d$ vertices, which we call $\lambda_1, \dots, \lambda_d$.

	Let also $v$ be a vertex of	$\Gamma$. We define $\tilde{\Gamma}$ to be the same graph as $\Gamma$ but removing all the loops on the vertex $v$. If $\tilde{\Gamma}_1, \dots, \tilde{\Gamma}_d$ are $d$ disjoint copies of $\tilde{\Gamma}$, with $v_1, \dots, v_d$ being the vertices corresponding to $v$ in $\Gamma$, we define the graph $\Star(\Lambda, \Gamma, v)$ as
	\[
	\Star(\Lambda, \Gamma, v) = \left( \Lambda \sqcup \bigsqcup_{i \in X} \tilde{\Gamma}_i \right) \big/ \{ \lambda_i = v_i \mid i \in X \}.
	\]

	Observe that no edges are identified in the process. All edges keep their labels. We will refer to $\tilde{\Gamma}_i$ as the $i$-th copy of $\Gamma$, slightly abusing notation. Every vertex in the $i$-th copy of $\Gamma$ is labeled by $wi$, if it corresponds to the vertex labeled by $w$ in $\Gamma$.
\end{definition}

This operation $\Star$ is a special case of the inflation of graphs used in~\cite{Bon07} and~\cite{BDN16}, but taking loops and labels into account.

Let now $G=G_\omega$, the spinal group defined by $d\ge 2$, $m\ge 1$ and $\omega\in \Omega_{d,m}$, as defined in Section 2. Recall that we consider it with spinal generating set  $S=S_\omega=A \cup B_\omega \setminus \{1\}$, and that we write $B$ for the spinal elements $B_\omega$.

\begin{proposition}
	\label{prop:description_Gamma_n}
	
	Let $\Gamma_n$ be the Schreier graph of a spinal group $G$ on the $n$-th level of the tree $T_d$, with respect to the spinal generating set $S$ . Then,
	\[
	\Gamma_1 = \Star(\Theta, \Xi, \emptyset), \qquad \Gamma_n = \Star(\Lambda_{\omega_{n-2}}, \Gamma_{n-1}, (d-1)^{n-2}0) \quad \forall n\ge 2, 
	\]
	where $\Xi$, $\Theta$ and $\Lambda_\pi$ are the following:
	
	\begin{itemize}
		\item $\Xi$ is a graph with one vertex, labeled $\emptyset$, which has $d^m-1$ loops, each labeled with a different, non-trivial element in $B$.
		\item $\Theta$ is a graph with $d$ vertices, labeled $0, \dots, d-1$. There is an edge labeled by $a^j$ from $i$ to $i + j$ (mod $d$), for every $i \in X$ and for every $1 \le j \le d-1$.
		\item If $\pi \in \Epi(B, A)$, $\Lambda_\pi$ is a graph with $d$ vertices, labeled $0, \dots, d-1$. For every $b \in B$, let $j$ be such that $\pi(b) = a^j$. Then, for every vertex $i$, there is an edge from $i$ to $i+j$ (mod $d$). Notice that this implies adding loops for all $b\in \Ker(\pi)$.
	\end{itemize}
	\input{img/blocks}
\end{proposition}

\begin{proof}
	First, recall from Section 2 that all elements in $B$ stabilize all vertices of the first level. This means that for every vertex in $\Gamma_1$, there are as many loops as elements in $B$. However, for every $1 \le j \le d-1$, $a^j$ maps vertex $i$ to $i+j$, for every $i \in X$. This proves the construction of $\Gamma_1$.
	
	Now assume $n \ge 2$, and let $\Gamma_n' = \Star(\Lambda_{\omega_{n-2}}, \Gamma_{n-1}, (d-1)^{n-2}0)$. By construction, we have $V(\Gamma_n) = V(\Gamma_n') = X^n$, so we only have to prove that the edges are the same. Let $v = v_0 \dots v_{n-1} \in X^n$, we will prove that its set of outgoing edges is the same in both $\Gamma_n$ and $\Gamma_n'$. Let us call $w = v_0 \dots v_{n-2} \in V(\Gamma_{n-1})$, so that $v = wv_{n-1}$.
	
	In $\Gamma_n$, it is clear that there is precisely one outgoing edge for every generator $s \in S$, going to $s(v)$. So, for every $s \in S$, we must prove that in $\Gamma_n'$ there is exactly one outgoing $s$-edge from $v$ to $s(v)$. Notice that $s(v) = s(w v_{n-1}) = s(w)v_{n-1}$ unless $w = (d-1)^{n-2}0$ and $s \in B$.

	Suppose first that $w \neq (d-1)^{n-2}0$. In that case, in $\Gamma_{n-1}$, $w$ has an outgoing $s$-edge towards $s(w)$ by hypothesis. Moreover, by construction, if $w \neq (d-1)^{n-2}0$, the outgoing $s$-edge from $v$ cannot go outside its copy of $\Gamma_{n-1}$. Therefore, $v$ must have an outgoing $s$-edge to $s(w)v_{n-1}$, and in fact $s(w)v_{n-1} = s(v)$, again because $w \neq (d-1)^{n-2}0$.

	Assume now that $w = (d-1)^{n-2}0$. If $s \in A$, $s(w) \neq w$, so the outgoing $s$-edge is not a loop, and so it is not removed in the construction of $\Gamma_n'$, hence there is an outgoing $s$-edge from $v$ to $s(w)v_{n-1}$ in $\Gamma_n'$. Because $s \in A$ changes only the first digit of any vertex, we have $s(v) = s(w)v_{n-1}$.

	Finally, suppose $w = (d-1)^{n-2}0$ and $s \in B$. In this case $s(w) = w$, so the edge is a loop and is indeed removed in the construction of $\Gamma_n'$. But the vertex $v$ is identified with the vertex labeled by $v_{n-1}$ in $\Lambda_{\omega_{n-2}}$, which has an outgoing $s$-edge towards the vertex labeled by $v_{n-1} + j$ in $\Lambda_{\omega_{n-2}}$, where $\omega_{n-2}(s) = a^j$. Therefore, in $\Gamma_n'$, the vertex $v$ has an outgoing $s$-edge towards the vertex $w(v_{n-1} + j)$. As it turns out, $s(v) = s((d-1)^{n-2}0v_{n-1}) = (d-1)^{n-2}0\omega_{n-2}(s)(v_{n-1}) = w(v_{n-1} + j)$, which completes the proof.
\end{proof}

\begin{remark}
\label{rem:diam}
	With this characterization, notice that $\diam(\Gamma_n) = 2^n - 1$.
\end{remark}

\section{Schreier graphs on $X^\N$}
\label{sec:schreier_infinite}


The action of a group $G$ on a rooted tree $T$ by automorphisms is naturally extended to the an action by homeomorphisms on the boundary $\partial T$ of the tree. In the case $T=T_d$, the boundary is in bijection with the set $X^\N$ of infinite words in the alphabet $X$. Hence we have infinite Schreier graphs $\Gamma_\xi = \Sch(G, G\cdot\xi, S) = \Sch(G, Stab_G(\xi), S)$, for $\xi \in X^\N$. 
In the case of spinal groups the extended action become, for any $\xi = \xi_0\xi_1\dots \in X^\N$, $a^j(\xi) = (\xi_0 + 1)\xi_1\dots$, and for every $b \in B$, $b(\xi)$ is the only word whose $n$-prefix is $b(\xi_0\dots\xi_{n-1})$ for every $n \ge 0$. 

Notice that $X^\N$ is endowed with the shift operator $\sigma: X^\N \to X^\N$, which removes the first letter of a sequence. We say that two points $\xi, \eta \in X^\N$ are \emph{cofinal} if they differ in finitely many letters, that is, if there exists some $r\ge0$ such that $\sigma^r(\xi) = \sigma^r(\eta)$. Cofinality is an equivalence relation, and we call the equivalence class of $\xi$ its \emph{cofinality class}.

\begin{proposition}
	Let $G$ be a spinal group and $\xi\in X^\N$. Then, $G\cdot \xi = \Cof(\xi)$.
\end{proposition}
\begin{proof}
	Because $a$ only changes the first digit of the sequence and any $b\in B$ either fixes the sequence or changes the first digit after a prefix $(d-1)^n0$, any generator changes at most one digit of the sequence. Let $\eta\in G\cdot \xi$, so there is a $g\in G$ such that $\eta = g\xi$. Then, $g$ changes at most $|g|_S$ digits in $\xi$, which implies that $\eta$ and $\xi$ are cofinal.
	
	Conversely, we can check that starting from $\xi$ and performing transformations corresponding to the generators (changing the first letter and changing the first letter after a specific pattern), we can obtain any point $\eta$ cofinal with $\xi$.
\end{proof}

\begin{proposition}
	\label{prop:description_Gamma_xi}
	Let $G$ be a spinal group and let $\xi\in X^\N$. The $\Gamma_\xi$ has vertex set $V(\Gamma_\xi) = \Cof(\xi)$ and the following outgoing edges for every $\eta = \eta_0\eta_1\dots \in Cof(\xi)$:
	
	\begin{itemize}
		\item $\forall j \in X\setminus\{0\}$, there is an edge to $(\eta_0 + j)\eta_1\dots$ labeled by $a^j$.
		
		\item $\forall b\in B$, if $\eta = (d-1)^n0\eta_{n+1}\sigma^{n+2}(\eta)$ for some $n \ge 0$, then there is an edge to $(d-1)^n0(\eta_{n+1} + j)\sigma^{n+2}(\eta)$ labeled by $b$, where $j$ is such that $\omega_n(b) = a^j$. Otherwise there is a loop at $\eta$ labeled by $b$.
	\end{itemize}
\end{proposition}
\begin{proof}
	We already know that the orbit of $\xi$ is its cofinality class, and the edges are easily checked using the definition of the action of $A$ and $B$ on $T_d$.
\end{proof}



\begin{definition}
	Let $\xi\in X^\N$ and $\eta \in \Cof(\xi)$. We define the following subgraphs of $\Gamma_\xi$:
	
	\begin{itemize}
		\item $\Delta_\eta^n = X^n\sigma^n(\eta) = \{w\sigma(\eta) \mid w \in X^n\}$, for $n \ge 0$, after removing all loops at $(d-1)^{n-1}0 \sigma^n(\eta)$, and also at $(d-1)^n \sigma^n(\eta)$ unless $\sigma^n(\eta)$ is fixed by $B$.
		\item $\Lambda_\eta^n = (d-1)^n0X\sigma^{n+2}(\eta) = \{(d-1)^n0i\sigma^{n+2}(\eta) \mid i \in X\}$, for $n \ge 0$.
	\end{itemize}
\end{definition}

Notice that $\eta \in \Delta_\eta^n$ for every $n \ge 0$, and $\Delta_\eta^n$ does not depend on $\eta_0\dots\eta_{n-1}$. For every $n \ge 0$, the subgraphs $\Delta_\eta^n$ cover all vertices of $\Gamma_\xi$ and have $d^n$ vertices. Moreover, no vertex in  $\Delta_\eta^n$ has outgoing edges to the rest of $\Gamma_\xi$ except for $(d-1)^{n-1}0 \sigma^n(\eta)$ and possibly $(d-1)^n \sigma^n(\eta)$, depending on whether $\sigma^n(\eta)$ is fixed or not by $B$.

On the other hand, $\eta$ does not belong to $\Lambda_\eta^n$ for any $n \ge 0$ if it is fixed by $B$, and belongs to exactly one $\Lambda_\eta^n$ if it is not. Moreover, $\Lambda_\eta^n$ does not depend on $\eta_0\dots\eta_{n+1}$, and is composed of $d$ vertices connected by only $B$-edges. If we ignore the vertex labels, $\Lambda_\eta^n$ is the same graph as $\Lambda_{\omega_n}$ in Proposition~\ref{prop:description_Gamma_n} and Fig.~\ref{fig:blocks}.

\begin{proposition}
	\label{prop:copy}
	
	Let $\xi \in X^\N$ and $\eta \in \Cof(\xi)$. Then, for every $n \ge 1$, the subgraph $\Delta_\eta^n$ of $\Gamma_\xi$ is isomorphic to the graph $\Gamma_n'$ obtained from $\Gamma_n$ by removing the following loops:
	\begin{itemize}
		\item All loops at the vertex $(d-1)^{n-1}0$.
		
		\item All loops at the vertex $(d-1)^n$  if $\sigma^n(\eta)$ is not fixed by $B$ .
	\end{itemize}

	In particular, the graph $\Delta_\eta^n$ is isomorphic to one of two graphs, only depending on whether $\sigma^n(\eta)$ is fixed by $B$ or not. The position of $\eta$ in this graph depends only on $\eta_0 \dots \eta_{n-1}$.
\end{proposition}
\begin{proof}
	Consider the bijection between the vertex sets given by $\varphi: v\sigma^n(\eta) \mapsto v$. Let $v \in X^n$ and $s \in S$, and let us prove that $v\sigma^n(\eta)$ has an outgoing $s$-edge to $s(v)\sigma^n(\eta)$ in $\Delta_\eta^n$ iff $v$ has an outgoing $s$-edge to $s(v)$ in $\Gamma_n'$.

	Suppose first that $v \neq (d-1)^{n-1}0, (d-1)^n$, so no loops are removed on $v$ in $\Gamma_n'$, and therefore $v$ has an outgoing $s$-edge to $s(v)$. For such $v$, we have $s(v \sigma^n(\eta)) = s(v) \sigma^n(\eta)$, and so $\varphi(s(v \sigma^n(\eta))) = \varphi(s(v) \sigma^n(\eta)) = s(v)$.

	Suppose now that $v = (d-1)^{n-1}0, (d-1)^n$. If $s \in A$, again $s(v \sigma^n(\eta)) = s(v) \sigma^n(\eta)$, and so $\varphi(s(v \sigma^n(\eta))) = \varphi(s(v) \sigma^n(\eta)) = s(v)$. Because $s(v) \neq v$, the edge is not a loop and hence is not removed in $\Gamma_n'$.

	Assume $s \in B$. If $v = (d-1)^{n-1}0$ or $v = (d-1)^n$, $v \sigma^n(\eta)$ has an outgoing $s$-edge to $s(v s|_v(\sigma^n(\eta)))$. This is either a loop or an edge to a vertex not in $\Delta_\eta^n$. If it is a loop, it is removed in both $\Delta_\eta^n$ and $\Gamma_n'$ iff $v = (d-1)^{n-1}0$ or $\sigma^n(\eta)$ is not fixed by $B$, and is kept in both otherwise. If it is an edge to the rest of $\Gamma_\xi$, in $\Gamma_n'$ it is a loop, which is removed as $\sigma^n(\eta)$ is not fixed by $B$. Hence there is no outgoing $s$-edge from $v$ or $v \sigma^n(\eta)$ in $\Gamma_n'$ or $\Delta_\eta^n$.
\end{proof}

We will say that these subgraphs $\Delta_\eta^n$ isomorphic to $\Gamma_n'$ are copies of $\Gamma_n$ in $\Gamma_\xi$. They are isomorphic up to some loops in only two vertices. 
We can regard $\Gamma_\xi$ as the disjoint union of copies of $\Delta_\eta^n$, glued together through copies of $\Lambda_\eta^r$, $r \ge n-1$.

\begin{remark}
	\label{rem:balls}
	Let $\xi \in X^\N$, $\eta \in \Cof(\xi)$ and $n \ge 0$. We denote by $\B_v(r)$ the ball centered at a vertex $v$ of radius $r \ge 0$. Combining Proposition~\ref{prop:copy} and Proposition~\ref{prop:description_Gamma_n}, we can check the following properties.

	\begin{equation}
	\bigcup_{v \in \Lambda_\eta^n} \B_v(2^{n+1} - 1) = \Delta_\eta^{n+2}.
	\end{equation}

	\begin{equation}
	\bigcup_{v \in \Lambda_\eta^n} \B_v(2^k - 1) = X^k(d-1)^{n-k}0X\sigma^{n+2}(\eta), \quad 0 \le k \le n.
	\end{equation}
\end{remark}

\section{Space of Schreier graphs}
\label{sec:space_schreier_graphs}

Recall that if $G$ is a group with a finite generating set $S$ acting on a topological space $\mathcal X$,  we can associate to the action the family of the rooted, directed, labeled orbital Schreier graph $(\Gamma_x, x)_{ x\in \mathcal X}$.  This defines a map $\Sch$ from $\mathcal X$ to the space of rooted, directed, labeled graphs $\mathcal{G}_{*, S}$:
\[
\begin{array}{llll}
\Sch: & \mathcal X & \to & \mathcal{G}_{*, S} \\
& x & \mapsto & (\Gamma_x,x)
\end{array}
\]
Note that we write $(\Gamma, v)$ for a graph $\Gamma$ rooted at the vertex $v$.
The space $\mathcal{G}_{*, S}$ is equipped with the pointed Gromov-Hausdorff topology. A basis for this topology is formed by the so-called cylinder sets, each of which contains all graphs with isomorphic $r$-balls around the root, for $r \ge 0$. If $G$ is a finitely generated group of automorphisms of a rooted spherically homogeneous tree $T$ that acts transitively of all levels of the tree  then, for each infinite ray $\xi\in\partial T$, the sequence of finite Schreier graphs $(\Gamma_n, \xi_n)$ converges to $(\Gamma_\xi, \xi)$, where $x_n$ is the vertex of the ray $\xi$ on the $n$-th level of the tree. 

Let now $G$ be a spinal group with the spinal generating set $S$, namely, $G=G_\omega$, $\omega\in \Omega_{d,m}$ with $d\geq 2$, $m\geq 1$, and $S=S_\omega$.

\begin{proposition}
	\label{prop:Sch_continuous}
	$\Sch$ is continuous everywhere except for points in $\Cof((d-1)^\N)$.
\end{proposition}
\begin{proof}
	Let $\xi$ be a point in the boundary, and let $\Sch(\xi)$ be its image. A neighborhood of $\Sch(\xi)$ is a set of rooted graphs such that, for some $r \ge 0$, their balls of radius $r$ are isomorphic to $\B = \B_\xi(r)$. Fix $r \ge 1$ and let $U$ be the corresponding neighborhood of $\Sch(\xi)$.
	
	Let now $R \ge 0$ be such that $\B \subset \Delta_\xi^R$ and such that $(d-1)^R\sigma^R(\xi) \not\in \B$. There is always an $R$ guaranteeing the first condition, but to ensure the second we need $\xi \not\in \Cof((d-1)^\N)$.
	
	Now consider the neighborhood $\xi_0 \dots \xi_{R-1}X^\N$ of $\xi$, and let $\eta$ be a point of this neighborhood. We need to show that $\Sch(\eta) \in U$, so that $\B' = \B_\eta(r)$ is isomorphic to $\B$. Using Proposition~\ref{prop:copy}, $\Delta_\xi^R$ and $\Delta_\eta^R$ are isomorphic to two possible graphs $\Gamma_R'$, which only differ in some loops at $(d-1)^R$. Since $(d-1)^R\sigma^R(\xi) \not \in \B$, and hence $(d-1)^R\sigma^R(\eta) \not \in \B'$ because the root is at the same position, $\B$ and $\B'$ must be isomorphic, and so $\Sch$ is continuous everywhere outside $\Cof((d-1)^\N)$.

	For points $\xi = w(d-1)^\N \in \Cof((d-1)^\N)$, consider the sequence $(w(d-1)^n0^\N)_n$, which converges to $\xi$. Any ball centered at $w(d-1)^n0^\N$ big enough will contain $(d-1)^{|w|+n}0^\N$, which is not fixed by $B$, and hence outgoing edges which are not loops for some $b \in B$, on that vertex, while $(d-1)^\N$ is fixed by $B$, and so all its outgoing $B$-edges are loops. The balls will not be isomorphic, and so $\Sch$ is not continuous on $\Cof((d-1)^\N)$.
\end{proof}


In order to transport the action of $G$ on $\mathcal X$ to the space of Schreier graphs, we thus have to consider $\overline{\Sch(X^\N \setminus \Cof((d-1)^\N))}$, as it is the closure of the image of the continuity points of $\Sch$. This will be done in Theorem~\ref{thm:description_space_graphs}, for which we will need the following Lemma.

\begin{lemma}
	\label{lem:common_prefix}
	Let $\xi, \eta \in X^\N$. If $B_\xi(r)$ and $B_\eta(r)$ are isomorphic, then $\xi$ and $\eta$ share a prefix of length $\lfloor\log_2(r)\rfloor$.
\end{lemma}
\begin{proof}
	Let $k = \lfloor\log_2(r)\rfloor$, so we have $r \ge 2^k$. By Proposition~\ref{prop:copy} we know that the subgraphs $\Delta_\xi^k$ and $\Delta_\eta^k$ are copies of $\Gamma_k$. Because the diameter of $\Gamma_k$ is $2^k - 1$ (see Remark~\ref{rem:diam}, they must be fully contained in $B_\xi(r)$ and $B_\eta(r)$, respectively. The isomorphism between the balls must then restrict to an isomorphism between $\Delta_\xi^k$ and $\Delta_\eta^k$, which maps $\xi$ to $\eta$. Both are mapped to the same vertex of $\Gamma_k$, which means that their prefixes of length $k$ coincide.
\end{proof}

\begin{theorem}
	\label{thm:description_space_graphs} Let $G = G_\omega$ be a spinal group with $\omega \in \Omega_{m,d}$ with $d=2,m\geq 2$ or $d\geq 3, m\geq 1$. Let $\Sch: X^\N \to \mathcal{G}_{*, S}$ be as above. Then
	\begin{enumerate}
		\item The map $\Sch$ is injective. It is continuous everywhere except in $\Cof((d-1)^\N)$.
		\item The set $\overline{\Sch(X^\N)}$ does not have isolated points, if $d > 2$. When $d=2$, the set of isolated points is $\Sch(\Cof((d-1)^\N))$.
		\item The set $\overline{\Sch(X^\N)}$ is the disjoint union of $\Sch(X^\N)$ and countably many points, which are obtained from finitely many $d$-ended graphs by choosing the root arbitrarily. These graphs are $\tilde{\Gamma}_\pi = \Star(\Lambda_\pi, \Gamma_{(d-1)^\N}, (d-1)^\N)$, for every $\pi \in \Epi(B, A)$ repeating infinitely often in $\omega$.
	\end{enumerate}
\end{theorem}

\begin{proof}
	Injectivity of $\Sch$ follows from the fact that spinal groups except $d=2,m=1$ are  branch groups \cite{BGS03}, and the stabilizers of boundary points are all different for the action of a branch group on the boundary of the tree (see Proposition 2.2. in \cite{Gri11}). In Proposition~\ref{prop:Sch_continuous} we found the continuity points of $\Sch$.
	
	Let $(\xi^{(n)})_n$ be a sequence of continuity points of $\Sch$ in $X^\N$ such that $(\Gamma_{\xi^{(n)}}, \xi^{(n)})$ converges to some rooted graph $(\Gamma, \xi)$. For any $r \ge 0$, there exists some $N \ge 0$ such that, for every $n \ge N$, the balls $\B_{\xi^{(n)}}(r)$ are all isomorphic. By Lemma~\ref{lem:common_prefix}, for every $n \ge N$, all $\xi^{(n)}$ share the same prefix of length $\lfloor \log_2(r) \rfloor$. Hence, $(\xi^{(n)})_n$ converges to a point $\xi^{(\infty)} \in X^\N$.
	
	If $\xi^{(\infty)}$ is a continuity point, then by continuity $(\Gamma, \xi) = (\Gamma_{\xi^{(\infty)}}, \xi^{(\infty)})$, so assume the opposite, which means that $\xi^{(\infty)} \in \Cof((d-1)^\N)$. By continuity of the action, we can assume without loss of generality that $\xi^{(\infty)} = (d-1)^\N$. Now two things can happen: either $\xi^{(n)}$ is fixed by $B$ for all large enough $n$ or not. If that is the case, then $(\Gamma, \xi) = (\Gamma_{(d-1)^\N}, (d-1)^\N)$, so graphs corresponding to points in $\Cof((d-1)^\N)$ are not isolated. Notice that this cannot happen if $d=2$, as all continuity points are not fixed by $B$, and so they cannot approximate $(d-1)^\N$.

	Suppose now that $\xi^{(n)}$ is not fixed by $B$ for all large enough $n$. This means that $\xi^{(n)}$ has a prefix $(d-1)^{k_n}0$, for all $n$ large enough. Because the balls $\B_{\xi^{(n)}}(1)$ must all be isomorphic, this means that every $b \in B$ must act in the same way for all points, so $\omega_{k_n}$ must be the same epimorphism $\pi \in \Epi(B, A)$ for all $n$ large enough. Therefore, the graph $(\Gamma, \xi)$ coincides with $(\tilde{\Gamma}_\pi, \eta)$, with $\eta$ being any vertex of $\Lambda_\pi$, as in the statement.
\end{proof}

\begin{remark}
	Notice that the graphs $\tilde{\Gamma}_\pi$ arising in Theorem~\ref{thm:description_space_graphs} have the product $\Cof((d-1)^\N) \times X$ as vertex set and the following edges:
	\begin{itemize}
		\item For every $s \in S$, $\xi \in \Cof((d-1)^\N)\setminus\{(d-1)^\N\}$ and $i \in X$, an $s$-edge from $(\xi, i)$ to $(s\xi, i)$.

		\item For every $1 \le j \le d-1$ and $i \in X$, an $a^j$-edge from $((d-1)^\N, i)$ to $(a^j(d-1)^\N, i)$.

		\item For every $b \in B$ and $i \in X$, a $b$-edge from $((d-1)^\N, i)$ to $((d-1)^\N, \pi(b)i)$.
	\end{itemize}

	If we consider the map $(\xi, i) \mapsto \xi$, it becomes clear that these graphs are $d$-coverings of $\Gamma_{(d-1)^\N}$. 
	While the graphs $\Gamma_\xi$ are Schreier graphs of subgroups $\Stab_G(\xi)$, for $m=1$ the graphs $\tilde{\Gamma}_\pi$ are Schreier graphs of $\Stab_G(N(\xi))$, the pointwise stabilizer of a neighborhood $N(\xi)$ of a point $\xi \in \Cof((d-1)^\N)$. More generally, for $m \ge 1$, the Schreier graphs associated to $\Stab_G(N(\xi))$ are $d^{m-1}$-coverings of the graphs $\tilde{\Gamma}_\pi$.
	
\end{remark}

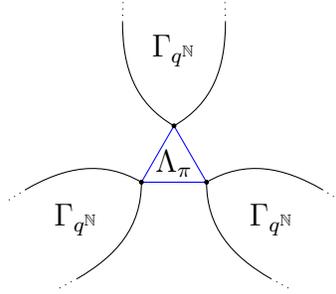
\begin{figure}[h]
	\centering
	\begin{tikzpicture}
		\pgfmathsetmacro{\r}{0.5}
		
		\draw[color=blue] (90:\r) -- (90 + 120:\r) -- (90 + 240:\r) -- cycle;
		\draw[fill=black] (90:\r) circle (0.05*\r);
		\draw[fill=black] (90 + 120:\r) circle (0.05*\r);
		\draw[fill=black] (90 + 240:\r) circle (0.05*\r);
		\draw (0,0) node {$\Lambda_\pi$};
		
		\draw (90:\r) to[out = 150, in=-90] (110:4*\r);
		\draw[dotted] (110:4*\r) -- (107.5:4.5*\r);
		\draw (90:\r) to[out = 30, in=-90] (70:4*\r);
		\draw[dotted] (70:4*\r) -- (72.5:4.5*\r);
		\draw (90:3*\r) node {$\Gamma_{q^{\N}}$};
		
		\draw (90 + 120 :\r) to[out = 150 + 120, in=-90 + 120] (110 + 120 :4*\r);
		\draw[dotted] (110 + 120:4*\r) -- (107.5 + 120:4.5*\r);
		\draw (90 + 120:\r) to[out = 30 + 120, in=-90 + 120] (70 + 120 :4*\r);
		\draw[dotted] (70 + 120:4*\r) -- (72.5 + 120:4.5*\r);
		\draw (90 + 120:3*\r) node {$\Gamma_{q^{\N}}$};
		
		\draw (90 + 240 :\r) to[out = 150 + 240, in=-90 + 240] (110 + 240 :4*\r);
		\draw[dotted] (110 + 240:4*\r) -- (107.5 + 240:4.5*\r);
		\draw (90 + 240:\r) to[out = 30 + 240, in=-90 + 240] (70 + 240 :4*\r);
		\draw[dotted] (70 + 240:4*\r) -- (72.5 + 240:4.5*\r);
		\draw (90 + 240:3*\r) node {$\Gamma_{q^{\N}}$};
	\end{tikzpicture}
  
	\caption{Graph $\tilde{\Gamma}_\pi$ for $d=3$.}
  
	\label{fig:Gamma_tilde_FG}
\end{figure}

\section{Number of ends of Schreier graphs}
\label{sec:number_of_ends}

When studying infinite graphs, an important invariant is the  number of ends. 

We say that a graph $\Gamma$ is $k$-ended if, for every vertex $v$, $\Gamma\setminus\{v\}$ has not more than $k$ infinite components, and this $k$ is minimal. 
The number of ends is a property of unlabeled, unrooted graphs, and loops or multiple edges do not play any role. Therefore, the number of ends of a Schreier graph $\Gamma_\xi$ of a spinal group depends only on $d$ and $\xi \in X^\N$, but not on $m$ or the sequence $\omega \in \Omega$. We can therefore count the number of ends for Schreier graphs of the spinal groups $G_d$ defined, for each $d\geq 2$, by $m = 1$ and $\omega$ the constant sequence. The groups $G_d$ happen to be automata groups, and we can apply the results on the ends of Schreier graphs of automata groups from \cite{BDN16}.

\begin{theorem}
\label{thm:number_ends}
Let $G_\omega$ be a spinal group and $\xi \in X^\N$. We partition $X^\N$ into $E_2 = X^*\{0,d-1\}^\N \setminus \Cof((d-1)^\N)$ and $E_1 = X^\N \setminus E_2$. Then $\Gamma_\xi$ is $2$-ended iff $\xi \in E_2$ and $\Gamma_\xi$ is $1$-ended iff $\xi \in E_1$.
\end{theorem}
\begin{proof}
As mentioned, it suffices to show the claim for the groups $G_d$, $d \ge 2$. $G_2$ is the infinite dihedral group, for which $X^\N$ not in $\Cof(1^\N)$ has trivial stabilizer, so all the corresponding Schreier graphs are isomorphic to the Cayley graph, a two-ended line alternating $a$ and $b$-edges.  Section 5.2 in~\cite{BDN16} shows it for the Fabrykowski-Gupta group $G_3$. The same proof applies for $d > 3$, by replacing $2$ by $d-1$ and $1$ by $\{1, \dots, d-2\}$.
\end{proof}

The set $X^\N$ is naturally equipped with the Bernoulli measure $\mu$. As the action of $G$ is ergodic with respect to $\mu$, we can speak about the typical number of ends.  

\begin{lemma}
	\label{lem:measures}
	If $E \subset X^\N$ such that $\mu(E) = 0$, then $\mu(X^*E) = 0$.
\end{lemma}
\begin{proof}
	We can decompose $X^*E = \underset{n \ge 0}{\bigsqcup} \: \underset{w \in X^n}{\bigsqcup} wE$, so
	\[
	\mu(X^*E) = \underset{n \ge 0}{\sum} \: \underset{w \in X^n}{\sum} \mu(wE) = \underset{n \ge 0}{\sum} \: \underset{w \in X^n}{\sum} \frac{\mu(E)}{d^n} = \underset{n \ge 0}{\sum} \mu(E) = 0.
	\]
\end{proof}

\begin{theorem}
	\label{thm:measures_ends}
	If $d = 2$, then $\mu(E_1) = 0$ and $\mu(E_2) = 1$. If $d \ge 3$, then $\mu(E_1) = 1$ and $\mu(E_2) = 0$.
\end{theorem}
\begin{proof}
	For $d=2$, we have $E_1 = \Cof(1^\N)$ and $E_2 = X^\N \setminus \Cof(1^\N)$. We use Lemma~\ref{lem:measures} with $E = \{1^\N\}$ provided that $E_1 = X^*E$.

	For $d \ge 3$, we have $\mu(E_2) \le \mu(X^*\{0, d-1\}^\N)$. Decomposing after the first digit, we obtain $\mu(\{0, d-1\}^\N) = \frac{2}{d} \mu(\{0, d-1\}^\N)$, so $\mu(\{0, d-1\}^\N) = 0$. Again by Lemma~\ref{lem:measures}, $\mu(E_2) = 0$.
\end{proof}

Theorem~\ref{thm:measures_ends} shows a difference between binary spinal groups and the rest. For spinal groups with $d = 2$, all Schreier graphs are two-ended except for one orbit, while for $d \ge 3$ the set of boundary points whose Schreier graph is one-ended has measure one, even though there are infinitely many orbits with two-ended graphs.

\section{Isomorphism classes of unlabeled Scheier graphs}
\label{sec:isomorphisms}

The goal of this section is to determine when two Schreier graphs $\Gamma_\xi$ and $\Gamma_\eta$, for $\xi, \eta \in X^\N$ are isomorphic. If we consider $(\Gamma_\xi, \xi)$ as rooted, directed, labeled graphs, the answer is immediate: if and only if 
$\xi=\eta$. For any branch group, Proposition 2.2 in~\cite{Gri11} implies that the stabilizers of different points of $X^\N$ are different. Since the graph $\Gamma_\xi$ is the Schreier graph of the group with respect to the stabilizer of the point $\xi$ for the action of $G_\omega$ on the boundary of tree, $(\Gamma_\xi, \xi)$ and $(\Gamma_\eta, \eta)$ will not be isomorphic as rooted, directed, labeled graphs, unless $\xi = \eta$.

The only exception to the previous statement is the infinite dihedral group, which is the spinal group with $d=2, m=1$ and the only spinal group which is not branch. As mentioned above, all points of $X^\N$ not in $\Cof(1^\N)$ give rise to Schreier graphs isomorphic to the Cayley graph, so in this case we do have isomorphic rooted, directed, labeled Schreier graphs. Some arguments we are going to use later do not apply to this case, so we will exclude this example in the forthcoming statements.

For the rest of the section, we will explore under which conditions two Schreier graphs can be isomorphic as rooted, undirected, unlabeled graphs.

\begin{proposition}
\label{prop:isomorphisms_d2}
Let $d = 2$, $m\geq 2$ and $\xi, \eta \in X^\N$. If $\xi \neq \eta$, $(\Gamma_\xi, \xi)$ and $(\Gamma_\eta, \eta)$ are isomorphic as rooted, unlabeled graphs iff $\xi, \eta \not \in \Cof(1^\N)$.
\end{proposition}
\begin{proof}
For $\xi \not \in \Cof(1^\N)$, $\Gamma_\xi$ is a two-ended line. Each vertex has $2^{m-1} - 1$ loops, one edge to one neighbor and $2^{m-1}$ edges to the other neighbor. The graph does not depend on $\xi$, so they are all isomorphic. If $\xi, \eta \in \Cof(1^\N)$, the distance to $1^\N$ is different for each of them, so no isomorphism is possible.
\end{proof}

Let us now desbribe the possible isomorphisms of rooted, undirected, unlabeled Schreier graphs for spinal groups with $d\geq 3$. 
 Let us start by giving some necessary conditions.

\begin{lemma}
	\label{lem:Delta_to_Delta}
	
	Let $\xi, \eta \in X^\N$, and let $\varphi: (\Gamma_\xi, \xi) \to (\Gamma_\eta, \eta)$ be an isomorphism of rooted, undirected, unlabeled graphs. Let also $\xi' \in \Cof(\xi)$, $\eta' = \varphi(\xi')$ and $n \ge 0$. Then, $\varphi(\Lambda_{\xi'}^n) = \Lambda_{\eta'}^n$ and $\varphi(\Delta_{\xi'}^n) = \Delta_{\eta'}^n$.
\end{lemma}
\begin{proof}
	We will abuse notation and denote subgraphs and their vertex sets in the same way. Notice that, even though we consider isomorphisms between unlabeled graphs, vertices that are fixed by $B$ must be mapped to vertices that are fixed by $B$, otherwise they have a different number of loops. The second statement is trivial for $n=0$, as $\xi'$ is mapped to $\eta' = \varphi(\xi')$. Let us show it for $n=1$. 

	We know that $\Delta_{\xi'}^1 = \{0\sigma(\xi'), \dots, (d-1)\sigma(\xi')\}$. Since $1\sigma(\xi')$ is fixed by $B$,  $\B_{1\sigma(\xi')}(1) = \Delta_{\xi'}^1$. Hence
	\[
	\varphi(\Delta_{\xi'}^1) = \varphi(\B_{1\sigma(\xi')}(1)) = \B_{\varphi(1\sigma(\xi'))}(1) \supset \Delta_{\varphi(1\sigma(\xi'))}^1 = \Delta_{\eta'}^1.
	\]
	The last equality comes from the fact that $\varphi(1\sigma(\xi'))$ is also fixed by $B$, and so can only be joined by an $A$-edge to $\eta'$. The other inclusion then follows from the cardinality of the sets.

	Let us now prove the first statement for all $n \ge 0$. Let $\xi'' \in \Lambda_{\xi'}^n$ and $\eta'' = \varphi(\xi'')$, therefore $\xi'' = (d-1)^n0i\sigma^{n+2}(\xi')$. Since it is not fixed by $B$, $\eta''$ also is not fixed by $B$, so $\eta'' \in \Lambda_{\eta'}^k$ for some $k \ge 0$. We decompose $\B_{\xi''}(1) = \Delta_{\xi''}^1 \sqcup \left( \Lambda_{\xi'}^n \setminus \{\xi''\} \right)$ and similarly $\B_{\eta''}(1) = \Delta_{\eta''}^1 \sqcup \left( \Lambda_{\eta'}^k \setminus \{\eta''\} \right)$. The isomorphism $\varphi$ maps one ball onto the other, but by the previous case, it maps $\Delta_{\xi''}^1$ to $\Delta_{\eta''}^1$. Hence, it must map $\Lambda_{\xi'}^n$ to $\Lambda_{\eta'}^k$. We can suppose without loss of generality that $k \le n$. To show $k = n$, let us suppose $k+1 \le n$ for a contradiction. Using Remark~\ref{rem:balls}, notice that
	\[
	\left| \bigcup_{v \in \Lambda_{\xi'}^n} \B_v (2^{k+1}) \right| = \left| \bigcup_{v \in \Lambda_{\xi'}^n} \B_v (2^{k+1} - 1) \right| + d(d-1) =
	\]\[
	= \left| X^{k+1}(d-1)^{n-k-1}0X\sigma^{n+2}(\xi') \right| + d(d-1) = d^{k+2} + d(d-1),
	\]
	while
	\[
	\left| \bigcup_{v \in \Lambda_{\eta'}^k} \B_v (2^{k+1}) \right| = \left| \bigcup_{v \in \Lambda_{\eta'}^k} \B_v (2^{k+1} - 1) \right| + p(d-1) =
	\]\[
	= \left| \Delta_{\eta'}^{k+2} \right| + p(d-1) = d^{k+2} + p(d-1),
	\]
	with $p = 1$ or $p = 2$ depending on whether $\sigma(\eta')$ is or is not fixed by $B$, respectively. In any case, one ball must be mapped onto the other, so $p = d \ge 3$, which is a contradiction. Therefore $\Lambda_{\xi'}^n$ must be mapped to $\Lambda_{\eta'}^n$.
	
	Finally, for $n \ge 2$ and again by Remark~\ref{rem:balls} we have
	\[
	\varphi(\Delta_{\xi'}^n) = \varphi\left( \bigcup_{v \in \Lambda_{\xi'}^{n-2}} \B_v (2^{n-1} - 1) \right) = \bigcup_{v \in \Lambda_{\xi'}^{n-2}} \varphi\left( \B_v (2^{n-1} - 1) \right) =
	\]\[
	= \bigcup_{v \in \Lambda_{\xi'}^{n-2}} \B_{\varphi(v)} (2^{n-1} - 1) = \bigcup_{v \in \Lambda_{\eta'}^{n-2}} \B_v (2^{n-1} - 1) = \Delta_{\eta'}^n.
	\]
\end{proof}

The important consequence of this Lemma is that, even if we look at unlabeled graphs, isomorphisms are not allowed to map $A$-edges to $B$-edges or viceversa. They are not allowed to map copies of $\Gamma_n$ to anything which is not a copy of $\Gamma_n$, which is quite restrictive.

\begin{lemma}
	\label{lem:same_zeros}
	Let $\xi, \eta \in X^\N$, and let $\varphi: (\Gamma_\xi, \xi) \to (\Gamma_\eta, \eta)$ be an isomorphism of rooted, unlabeled graphs. For every $n \ge 0$, $\xi_n = 0 \Longleftrightarrow \eta_n = 0$.
\end{lemma}
\begin{proof}
	The case $n=0$ is a consequence of the first statement in Lemma~\ref{lem:Delta_to_Delta}. Let $n \ge 1$ and let $\xi_n = 0$.
	
	The second statement in the same Lemma implies that $\Delta_\xi^n = X^n0\sigma^{n+1}(\xi)$ is mapped to $\Delta_\eta^n = X^n\eta_n\sigma^{n+1}(\eta)$. If $\eta_n = d-1$, then the latter has a vertex in $\Lambda_{\eta'}^k$, for some $\eta' \in \Cof(\eta)$ and $k \ge n+1$, while the former does not have any vertex in $\Lambda_{\xi'}^k$ for any $\xi' \in \Cof(\xi)$ and $k \ge n+1$. By Lemma~\ref{lem:Delta_to_Delta}, this is a contradiction. Suppose $\eta_n \neq 0, d-1$. Because vertices $(d-1)^n0\sigma^{n+1}(\xi)$ and $(d-1)^n\eta_n\sigma^{n+1}(\eta)$ is and is not fixed by $B$, respectively, by Remark~\ref{rem:balls} we have
	\[
	\left| \bigcup_{v \in \Lambda_\xi^{n-2}} \B_v (2^{n-1}) \right| = \left| \bigcup_{v \in \Lambda_\xi^{n-2}} \B_v (2^{n-1} - 1) \right| + 2(d-1) =
	\]\[
	= \left| \Delta_\xi^n \right| + 2(d-1) = d^n + 2(d-1),
	\]

	\[
	\left| \bigcup_{v \in \Lambda_\eta^{n-2}} \B_v (2^{n-1}) \right| = \left| \bigcup_{v \in \Lambda_\eta^{n-2}} \B_v (2^{n-1} - 1) \right| + d - 1 =
	\]\[
	= \left| \Delta_\xi^n \right| + d - 1 = d^n + d - 1,
	\]
	which is also a contradiction. Hence $\eta_n = 0$.
\end{proof}

Lemma~\ref{lem:same_zeros} states that for two infinite rays to have isomorphic rooted, undirected, unlabeled Schreier graphs they must have  zeros in the same positions. This excludes any isomorphism between graphs of points with finitely and infinitely many zeros. Moreover, it allows us to write the image of $\xi = w_0 0 w_1 0 \dots$, with $w_k \in (X\setminus\{0\})^*$ as $\varphi(\xi) = \tilde{w}_0 0 \tilde{w}_1 0 \dots$, with $|w_k| = |\tilde{w}_k|$ for every $k \ge 0$.

Let $n \ge 0$ and $0 \le m \le n-1$. The sets
\[
Y_m = (X\setminus\{0\})^{n-m-1}(X\setminus \{0,d-1\})(d-1)^m, \quad Y_n = \{(d-1)^n\}
\]
define a partition of the set $(X\setminus\{0\})^n = \sqcup_{m=0}^n Y_m$.

\begin{definition}
	\label{def:compatibility}
	We call $\xi, \eta \in X^\N$ \emph{compatible} (and we denote this by $\xi \sim \eta$) if they are of the form $\xi = w_0 0 w_1 0 \dots$, $\eta = \tilde{w}_0 0 \tilde{w}_1 0 \dots$, with $w_k, \tilde{w}_k \in (X\setminus\{0\})^*\sqcup (X\setminus\{0\})^\N$, $|w_k| = |\tilde{w}_k|$, and $w_k \in Y_{m_k}$ iff $\tilde{w}_k \in Y_{m_k}$, for every $k \ge 0$ such that $|w_k| < \infty$. Notice that compatibility is an equivalence relation on $X^\N$.
\end{definition}

\begin{proposition}
	\label{prop:isomorphism_necessary}
	Let $\xi, \eta \in X^\N$, and let $\varphi: (\Gamma_\xi, \xi) \to (\Gamma_{\eta}, \eta)$ be an isomorphism of rooted unlabeled graphs. Then, $\xi \sim \eta$.
\end{proposition}
\begin{proof}
	By Lemma~\ref{lem:same_zeros}, we can assume that $\xi = w_0 0 w_1 0 \dots$ and $\eta = \tilde{w}_0 0 \tilde{w}_1 0 \dots$ with $w_k, \tilde{w}_k \in (X\setminus\{0\})^*\sqcup (X\setminus\{0\})^\N$ and $|w_k| = |\tilde{w}_k|$ for every $k \ge 0$ such that $|w_k| < \infty$. Let us assume, for a contradiction, that there exists some $k \ge 0$ such that, if $n=|w_k|$, $w_k \in Y_m$ and $\tilde{w}_k \in Y_{m'}$, with $m < m'$. Let also $N_r = \sum_{s=0}^{r-1} (|w_s| + 1)$, so that $\sigma^{N_r}(\xi) = w_k0w_{k+1}0\dots$ and $\sigma^{N_r}(\eta) = \tilde{w}_k0\tilde{w}_{k+1}0\dots$.
	
	Assume first $m < m' = n$, then $w_k = wi(d-1)^m$ and $\tilde{w}_k = (d-1)^n$, for some $i\in X\setminus\{0,d-1\}$ and $w \in (X\setminus\{0\})^{n - m - 1}$. We exclude the case $k = 0$ because then $\xi$ and $\eta$ would respectively be fixed and not fixed by $B$, which is absurd. So $N_k > 0$ and let us define the following copies of $\Gamma_{N_k}$, which respectively contain $\xi$ and $\eta$:
	\[
	\Delta_1 = X^{N_k}wi(d-1)^m0\sigma^{N_{k+1}}(\xi),
	\qquad
	\Delta_2 = X^{N_k}(d-1)^n0\sigma^{N_{k+1}}(\eta).
	\]
	
	By Lemma~\ref{lem:Delta_to_Delta}, we have $\varphi(\Delta_1) = \Delta_2$. The latter contains the vertex $(d~-~1)^{N_k + n}0\sigma{N_{n+1}}(\eta)$ which belongs to $\Lambda_{\eta'}^{N_k + n}$, for some $\eta' \in \Cof(\eta)$. However, the former does not contain any vertex which belongs to $\Lambda_{\xi'}^M$ for any $\xi' \in \Cof(\xi)$ and $M \ge N_k$. Again by Lemma~\ref{lem:Delta_to_Delta}, this is a contradiction.

	If $m < m' < n$, then $m < n - 1$ or equivalently $n - m - 1 > 0$. In this case, we consider the following copies of $\Gamma_{N_k + n - m - 1}$, again containing $\xi$ and $\eta$, respectively:
	\[
	\Delta_1 = X^{N_k + n - m - 1}i(d-1)^m0\sigma^{N_{k+1}}(\xi),
	\]\[
	\Delta_2 = X^{N_k + n - m - 1}(d-1)^{m+1}0\sigma^{N_{k+1}}(\eta).
	\]
	
	Again Lemma~\ref{lem:Delta_to_Delta} implies $\varphi(\Delta_1) = \Delta_2$. Similarly to the previous case, the latter contains the vertex $(d-1)^{N_k + n}0\sigma^{N_{n+1}}(\eta)$, which belongs to $\Lambda_{\eta'}^{N_k + n}$ for some $\eta' \in \Cof(\eta)$. In the former, there is not any vertex which belongs to $\Lambda_{\xi'}^M$, for any $\xi' \in \Cof(\xi)$ and $M \ge N_k + n - m - 1$. As $n - m - 1 > 0$, by Lemma~\ref{lem:Delta_to_Delta} this is a contradiction.
\end{proof}

We have shown that compatibility is a necessary condition to find an isomorphism of rooted, unlabeled graphs. Our next goal will be to prove the converse. Namely, that for any two compatible points there exists an isomorphism mapping their graphs to each other. To this purpose, let use introduce the following notation.

\begin{definition}
	\label{def:R}
	Let $\xi, \eta \in X^\N$. We define $R = R_{\xi, \eta} = \min \{s \mid \sigma^s(\xi) = \sigma^s(\eta)\}$. Notice that $R < \infty$ iff $\eta \in \Cof(\xi)$. Also note that $R = 0$ iff $\eta = \xi$ and that otherwise, by minimality, $\eta_{R - 1} \neq \xi_{R - 1}$.
\end{definition}

\begin{definition}
	Let $\xi, \eta \in X^\N$. We define $\tau_n = \tau_{n, \xi, \eta} = (\xi_n \: \eta_n) \in Sym(X)$. Let us also define $\varphi = \varphi_{\xi, \eta}: (\Gamma_\xi, \xi) \to (\Gamma_{\eta}, \eta)$ as
	\[
	\varphi(\xi') = \left\{
	\begin{array}{ll}
	\eta & \quad \text{if } \xi' = \xi\\
	\xi_0'\dots \xi_{R - 2}'\tau_{R - 1}(\xi_{R - 1}')\sigma^R(\eta) & \quad \text{if } \xi' \neq \xi, \quad \text{with } R = R_{\xi, \xi'}
	\end{array}
	\right..
	\]
\end{definition}

\begin{remark}
\label{rem:isomorphisms}
Suppose $\xi \sim \xi'$, and let $\eta' = \varphi(\xi')$.
\begin{enumerate}
	\item $R_{\eta, \eta'} = R_{\xi, \xi'}$. It is clear that $R_{\eta, \eta'} \le R_{\xi, \xi'}$, and if it was strictly smaller, then, setting $n = R_{\xi, \xi'} - 1$, we would have $\eta_n = \tau_n(\xi_n')$, and so $\xi_n = \xi_n'$, which is a contradiction with the minimality of $R_{\xi, \xi'}$.

	\item $\psi = \varphi_{\eta, \xi}$ is the inverse of $\varphi = \varphi_{\xi, \eta}$. Indeed, $\psi(\varphi(\xi)) = \xi$, and, for $\xi' \neq \xi$,
	\[
	\psi(\varphi(\xi')) = \psi\left(\xi_0'\dots \xi_{R - 2}' \tau_{R-1}(\xi_{R - 1}')\sigma^R(\eta)\right) =
	\]\[
	= \xi_0'\dots \xi_{R - 2}' \tau_{R-1}^2(\xi_{R - 1}')\sigma^R(\xi) = \xi_0'\dots \xi_{R - 2}'\xi_{R - 1}'\sigma^R(\xi) = \xi'.
	\]
	Hence, $\varphi$ is a bijection between the vertex sets of $\Gamma_\xi$ and $\Gamma_\eta$.
\end{enumerate}
\end{remark}

\begin{proposition}
	\label{prop:isomorphism_sufficient}
	Let $\xi, \eta \in X^\N$ such that $\xi \sim \eta$. Then $\varphi_{\xi, \eta}: (\Gamma_\xi, \xi) \to (\Gamma_\eta, \eta)$ is an isomorphism of rooted, unlabeled graphs.
\end{proposition}
\begin{proof}
	Let $\varphi = \varphi_{\xi, \eta}$. We already proved in Remark~\ref{rem:isomorphisms} that $\varphi$ is a bijection between the sets of vertices, so we only have to prove that edges are preserved. Every $A$ or $B$-edge belongs to $\Delta_{\xi'}^1$ or $\Lambda_{\xi'}^n$, respectively, for some $\xi' \in \Cof(\xi)$ and $n \ge 0$. Let then $\xi' \in \Cof(\xi)$, $n \ge 0$, set $\eta' = \varphi(\xi')$ and let us prove $\varphi(\Delta_{\xi'}^1) = \Delta_{\eta'}^1$ and $\varphi(\Lambda_{\xi'}^n) = \Lambda_{\eta'}^n$.
	
	For the first claim, let us consider two cases. If $\xi \in \Delta_{\xi'}^1$, then we have
	\[
	\varphi(\Delta_{\xi'}^1) = \varphi(\Delta_\xi^1) = \varphi\left(\{\xi\} \sqcup \{  j\sigma(\xi) \mid j \neq \xi_0\}\right) =
	\]\[
	= \{\eta\} \sqcup \{\tau_0(j)\sigma(\eta) \mid j \neq \xi_0\} =\{\eta\} \sqcup \{j\sigma(\eta) \mid j \neq \eta_0\} = \Delta_{\eta}^1.
	\]
	Since $R_{\xi, \xi'} \le 1$, by Remark~\ref{rem:isomorphisms} $R_{\eta, \eta'} \le 1$, so $\Delta_{\eta}^1 = \Delta_{\eta'}^1$. 

	Assume now $\xi \not\in \Delta_{\xi'}^1$, so $R = R_{\xi, \xi'} \ge 2$. Also, $R_{\xi, v} = R$ for all $v \in \Delta_{\xi'}^1$. Then
	\[
	\varphi(\Delta_{\xi'}^1) = \varphi(\{j\xi_1'\dots\xi_{R-1}'\sigma^R(\xi) \mid j \in X\}) = \]\[
	= \{j\xi_1'\dots\xi_{R-2}'\tau_{R-1}(\xi_{R-1}')\sigma^R(\eta) \mid j \in X \} = \Delta_{\eta'}^1.
	\]
	
	To prove the second claim, set $R = \min_{v \in \Lambda_{\xi'}^n} R_{\xi, v}$. In this case, $R$ cannot be $n+2$, since $v_{n+1}$ takes all values in $X$ when $v$ runs through $\Lambda_{\xi'}^n$. We will consider four different cases, depending on the value of $R$.

	Suppose first $R = 0$. In this case, $\xi \in \Lambda_{\xi'}^n$. Moreover, $R_{\xi, v} = n + 2$ for all $v \in \Lambda_{\xi'}^n$ except for $v = \xi$. Therefore,
	\[
	\varphi(\Lambda_{\xi'}^n) = \varphi(\{\xi\} \sqcup \{(d-1)^n0j\sigma^{n+2}(\xi) \mid j \neq \xi_{n+1} \}) =
	\]\[
	= \{\eta\} \sqcup \{(d-1)^n0\tau_{n+1}(j)\sigma^{n+2}(\eta) \mid j \neq \xi_{n+1} \} =
	\]\[
	= \{\eta\} \sqcup \{(d-1)^n0j\sigma^{n+2}(\eta) \mid j \neq \eta_{n+1} \} =
	\]\[
	= \{(d-1)^n0j\sigma^{n+2}(\eta) \mid j \in X\} = \Lambda_{\eta}^n.
	\]
	And $\Lambda_{\eta}^n = \Lambda_{\eta'}^n$, because $R_{\eta, \eta'} = R_{\xi, \xi'} \le n+2$, so $\sigma^{n+2}(\eta) = \sigma^{n+2}(\eta')$.

	Suppose now $1 \le R \le n$. For the vertex $v = (d-1)^n0\xi_{n+1}\sigma^{n+2}(\xi') \in \Lambda_{\xi'}^n$ we must have $R_{\xi, v} = R \le n$. This implies $\xi = \xi_0 \dots \xi_{R-1} (d-1)^{n-R}0\xi_{n+1}\sigma^{n+2}(\xi)$ and $\sigma^{n+2}(\xi) = \sigma^{n+2}(\xi')$. Additionally, since $\xi \sim \eta$, we can write $\eta = \eta_0 \dots \eta_{R-1}(d~-~1)^{n-R}0 \eta_{n+1}\sigma^{n+2}(\eta)$. For every $w \in \Lambda_{\xi'}^n$, $w \neq v$, we have $R_{\xi, w} = n + 2$. Then,
	\[
	\varphi(\Lambda_{\xi'}^n) = \varphi\left(\{v\} \sqcup \{(d-1)^n0j\sigma^{n+2}(\xi) \mid j \neq \xi_{n+1}\}\right) =
	\]\[
	= \{(d-1)^{R-1}\tau_{R-1}(d-1)(d-1)^{n-R}0\eta_{n+1}\sigma^{n+2}(\eta)\} \sqcup
	\]\[
	\sqcup \{(d-1)^n0\tau_{n+1}(j)\sigma^{n+2}(\eta) \mid j \neq \xi_{n+1}\} =
	\]\[
	= \{(d-1)^n0\eta_{n+1}\sigma^{n+2}(\eta)\} \sqcup \{(d-1)^n0j\sigma^{n+2}(\xi) \mid j \neq \eta_{n+1}\} = \Lambda_{\eta}^n.
	\]
	where we used $\tau_{R - 1}(d-1) = d-1$ because $\xi \sim \eta$. In addition, again $R_{\eta, \eta'} = R_{\xi, \xi'} \le n + 2$, so $\Lambda_{\eta}^n = \Lambda_{\eta'}^n$.

	The third case is $R = n + 1$. The vertex $v = (d-1)^n0\xi_{n+1}\sigma^{n+2}(\xi') \in \Lambda_{\xi'}^n$ satisfies $R_{\xi, v} = R = n+1$. This implies $\sigma^{n+2}(\xi) = \sigma^{n+2}(\xi')$. For any other $w \in \Lambda_{\xi'}^n$, we must have $R_{\xi, w} = n + 2$. Therefore,
	\[
	\varphi(\Lambda_{\xi'}^n) = \varphi(\{v\} \sqcup \{(d-1)^n0j\sigma^{n+2}(\xi) \mid j \neq \xi_{n+1}\}) =
	\]\[
	= \{(d-1)^n\tau_n(0)\eta_{n+1}\sigma^{n+2}(\eta)\} \sqcup \{(d-1)^n0\tau_{n+1}(j)\sigma^{n+2}(\eta) \mid j \neq \xi_{n+1}\} =
	\]\[
	= \{(d-1)^n0\eta_{n+1}\sigma^{n+2}(\eta)\} \sqcup \{(d-1)^n0j\sigma^{n+2}(\eta) \mid j \neq \eta_{n+1}\} = \Lambda_{\eta}^n,
	\]
	where we used $\tau_n(0) = 0$ after Lemma~\ref{lem:same_zeros} and provided that $\xi \sim \eta$. Once more, $R_{\eta, \eta'} = R_{\xi, \xi'} \le n + 2$, so $\Lambda_{\eta}^n = \Lambda_{\eta'}^n$.

	Finally, assume $R \ge n + 3$. We can write $\Lambda_{\xi'}^n = \{(d-1)^n0j\sigma^{n+2}(\xi') \mid j \in X\} = \{(d-1)^n0j\xi_{n+2}'\dots\xi_{R-1}'\sigma^R(\xi') \mid j \in X\}$. In this case, $R_{\xi, v} = R$ for every $v \in \Lambda_{\xi'}^n$, and also $R_{\xi, \xi'} = R$, which implies $\sigma^R(\xi) = \sigma^R(\xi')$. Then,
	\[
	\varphi(\Lambda_{\xi'}^n) = \varphi(\{(d-1)^n0j\xi_{n+2}'\dots\xi_{R-1}'\sigma^R(\xi) \mid j \in X\}) =
	\]\[
	= \{(d-1)^n0j\xi_{n+2}'\dots \xi_{R-2}'\tau_{R-1}(\xi_{R-1}')\sigma^R(\eta) \mid j \in X\} =
	\]\[ 
	= \Lambda_{\xi_0'\dots\xi_{R-2}'\tau_{R-1}(\xi_{R-1}')\sigma^R(\eta)}^n,
	\]
	and notice that $\eta' = \xi_0'\dots\xi_{R-2}'\tau_{R-1}(\xi_{R-1}')\sigma^R(\eta)$, so
	\[
	\Lambda_{\xi_0'\dots\xi_{R-2}'\tau_{R-1}(\xi_{R-1}')\sigma^R(\eta)}^n = \Lambda_{\eta'}^n.
	\]
\end{proof}

\begin{theorem}
	\label{thm:rooted_isomorphism}
	Let $\xi, \eta \in X^\N$. Then, $(\Gamma_\xi, \xi)$ and $(\Gamma_\eta, \eta)$ are isomorphic as rooted, undirected, unlabeled graphs iff $\xi \sim \eta$.
\end{theorem}
\begin{proof}
	One implication is Proposition~\ref{prop:isomorphism_necessary} and the other is Proposition~\ref{prop:isomorphism_sufficient}.
\end{proof}

We can now adapt Theorem~\ref{thm:rooted_isomorphism} to unrooted graphs.

\begin{corollary}
	\label{cor:unrooted_isomorphism}
	Let $\xi, \eta \in X^\N$. Then, $\Gamma_\xi$ and $\Gamma_\eta$ are isomorphic as unrooted, undirected, unlabeled graphs iff there exists $\eta' \in \Cof(\eta)$ such that $\xi \sim \eta'$.
\end{corollary}
\begin{proof}
	Let $\varphi: \Gamma_\xi \to \Gamma_\eta$ be an isomorphism of unrooted graphs. Then $\eta' = \varphi(\xi) \in \Cof(\eta)$, and $\tilde{\varphi}: (\Gamma_\xi, \xi) \to (\Gamma_\eta, \eta')$ is an isomorphism of rooted graphs. By Theorem~\ref{thm:rooted_isomorphism}, $\xi \sim \eta'$.

	On the other hand, if there exists $\eta' \in \Cof(\eta)$ such that $\xi \sim \eta'$, again by Theorem~\ref{thm:rooted_isomorphism} $(\Gamma_\xi, \xi)$ and $(\Gamma_\eta, \eta')$ are isomorphic as rooted graphs, and so as unrooted graphs as well.
\end{proof}

\begin{theorem}
	\label{thm:measure_isomorphism_classes}
	All rooted, unlabeled isomorphism classes have measure zero in $X^\N$.
\end{theorem}
\begin{proof}
	Let $\xi \in X^\N$, and let $C = \Comp(\xi) \subset X^\N$ be its compatibility class. We will show $\mu(C) = 0$.

	Suppose first that $\xi$ has finitely many zeros. In that case, there exists $N \ge 0$, $w \in X^N$ and $\xi' \in (X\setminus\{0\})^\N$ such that $\xi = w \xi'$. Moreover, $C \subset X^N (X\setminus\{0\})^\N$, so
	\[
	\mu(C) \le \mu(X^N (X\setminus\{0\})^\N) = \mu((X\setminus\{0\})^\N) = 0.
	\]

	Assume that $\xi$ has infinitely many zeros, so there is an infinite sequence of words $(w_k)_k$, $w_k \in (X\setminus\{0\})^*$ such that $\xi = w_0 0 w_1 0 \dots$. Let $n_k = |w_k|$ and $C_k = \Comp(w_k 0 w_{k+1} 0 \dots)$. By Lemma~\ref{lem:same_zeros}, any compatible point must have the same zeros at the same positions, so we have
	\[
	\mu(C) \le \mu(X^{n_0}0C_1) = \frac{1}{d}\mu(C_1).
	\]
	If we iterate this inequality, we have that, for any $k \ge 0$,
	\[
	\mu(C) \le \frac{1}{d^k}\mu(C_k).
	\]
	In particular, for any $\varepsilon > 0$, by choosing $k$ so that $\frac{1}{d^k} \le \varepsilon$, $\mu(C) \le \frac{1}{d^k} \mu(C_k) \le \varepsilon$. Hence, $\mu(C) = 0$.
\end{proof}


\begin{corollary}
	All unrooted, unlabeled isomorphism classes have measure zero in $X^\N$.
\end{corollary}
\begin{proof}
	Let $\xi \in X^\N$, and let $\Comp(\xi) \subset X^\N$ be its compatibility class. Let also $D = \cup_{g \in G} \Comp(g \xi)$ be it so-called eventual compatibility class. After Corollary~\ref{cor:unrooted_isomorphism}, $\Gamma_\xi$ is isomorphic to $\Gamma_\eta$ iff $\eta \in D$. Because the action preserves the measure $\mu$ and $\mu(\Comp(\xi)) = 0$ after Theorem~\ref{thm:measure_isomorphism_classes}, we have
	\[
	\mu(D) \le \sum_{g \in G}  \mu(\Comp(g\xi)) = \sum_{g \in G}  \mu(\Comp(\xi)) = 0.
	\]
\end{proof}

\section{Examples}
\label{sec:examples}

\subsection{Grigorchuk groups $G_\omega$}

Spinal groups were introduced as a generalization for Grigorchuk groups. As mentioned in Section 2, these correspond to $d=2$ and $m=2$. From our description we easily recover the construction of the Schreier graphs for Grigorchuk groups given in~\cite{BG00} for the first Griogrhuk group and in~\cite{Mat15} for all the groups $G_\omega$ with $d=2,m=2$. The graphs $\Xi$, $\Theta$ and $\Lambda_{\omega_n}$ are defined in Fig.~\ref{fig:blocks}, and there is an example of $\Gamma_n$ in Fig.~\ref{fig:Gamma3_Grig}.
	
\begin{figure}[h]
    \centering
    \begin{tikzpicture}
	    \pgfmathsetmacro{\r}{1.2}
	    
	    \begin{scope}[shift={(0,0)}]
		    \pic {Loop={(0,0)}{0}{\r}{blue}};
		    \pic {Loop={(0,0)}{90}{\r}{red}};
		    \pic {Loop={(0,0)}{180}{\r}{green}};
		    \draw (-0.9*\r, 0) node {\tiny{$c$}};
		    \draw (0, 0.9*\r) node {\tiny{$b$}};
		    \draw (0, -0.9*\r) node {\tiny{$d$}};
	    \end{scope}
	    
	    \begin{scope}[shift={(0,0)}]
		    \pic {Line={(0,0)}{0}{\r}{black}};
		    \draw (0.5*\r,0) node[above] {\tiny{$a$}};
	    \end{scope}
	    
	    \begin{scope}[shift={(\r,0)}]
		    \pic {Loop={(0,0)}{0}{\r}{green}};
		    \pic {Loop={(\r,0)}{0}{\r}{green}};
		    \draw[color=blue] (0,0) .. controls (0.5*\r, 0.2*\r) .. (\r,0);
		    \draw[color=red] (0,0) .. controls (0.5*\r, -0.2*\r) .. (\r,0);

		    \draw (0.5*\r, 0.1*\r) node[above] {\tiny{$b$}};
		    \draw (0.5*\r, -0.1*\r) node[below] {\tiny{$c$}};
		    \draw (0, 0.9*\r) node {\tiny{$d$}};
		    \draw (\r, 0.9*\r) node {\tiny{$d$}};
	    \end{scope}
	    
	    \begin{scope}[shift={(2*\r,0)}]
		    \pic {Line={(0,0)}{0}{\r}{black}};
		    \draw (0.5*\r,0) node[above] {\tiny{$a$}};
	    \end{scope}
	    
	    \begin{scope}[shift={(3*\r,0)}]
		    \pic {Loop={(0,0)}{0}{\r}{red}};
		    \pic {Loop={(\r,0)}{0}{\r}{red}};
		    \draw[color=blue] (0,0) .. controls (0.5*\r, 0.2*\r) .. (\r,0);
		    \draw[color=green] (0,0) .. controls (0.5*\r, -0.2*\r) .. (\r,0);

		    \draw (0.5*\r, 0.1*\r) node[above] {\tiny{$b$}};
		    \draw (0.5*\r, -0.1*\r) node[below] {\tiny{$d$}};
		    \draw (0, 0.9*\r) node {\tiny{$c$}};
		    \draw (\r, 0.9*\r) node {\tiny{$c$}};
	    \end{scope}
    
	    \begin{scope}[shift={(4*\r,0)}]
		    \pic {Line={(0,0)}{0}{\r}{black}};
		    \draw (0.5*\r,0) node[above] {\tiny{$a$}};
	    \end{scope}
	    
	    \begin{scope}[shift={(5*\r,0)}]
		    \pic {Loop={(0,0)}{0}{\r}{green}};
		    \pic {Loop={(\r,0)}{0}{\r}{green}};
		    \draw[color=blue] (0,0) .. controls (0.5*\r, 0.2*\r) .. (\r,0);
		    \draw[color=red] (0,0) .. controls (0.5*\r, -0.2*\r) .. (\r,0);

		    \draw (0.5*\r, 0.1*\r) node[above] {\tiny{$b$}};
		    \draw (0.5*\r, -0.1*\r) node[below] {\tiny{$c$}};
		    \draw (0, 0.9*\r) node {\tiny{$d$}};
		    \draw (\r, 0.9*\r) node {\tiny{$d$}};
	    \end{scope}
	    
	    \begin{scope}[shift={(6*\r,0)}]
		    \pic {Line={(0,0)}{0}{\r}{black}};
		    \draw (0.5*\r,0) node[above] {\tiny{$a$}};
	    \end{scope}
	    
	    \begin{scope}[shift={(7*\r,0)}]
		    \pic {Loop={(0,0)}{0}{\r}{blue}};
		    \pic {Loop={(0,0)}{-90}{\r}{red}};
		    \pic {Loop={(0,0)}{180}{\r}{green}};
		    \draw (0.9*\r, 0) node {\tiny{$c$}};
		    \draw (0, 0.9*\r) node {\tiny{$b$}};
		    \draw (0, -0.9*\r) node {\tiny{$d$}};
	    \end{scope}
	    
	    \draw (0,0) node[below] {\tiny{$110$}};
	    \draw (\r,0) node[below] {\tiny{$010$}};
	    \draw (2*\r,0) node[below] {\tiny{$000$}};
	    \draw (3*\r,0) node[below] {\tiny{$100$}};
	    \draw (4*\r,0) node[below] {\tiny{$101$}};
	    \draw (5*\r,0) node[below] {\tiny{$001$}};
	    \draw (6*\r,0) node[below] {\tiny{$011$}};
	    \draw (7*\r,0) node[below] {\tiny{$111$}};
		    
	\end{tikzpicture}
  \caption{$\Gamma_3$ for Grigorchuk's group.}
  \label{fig:Gamma3_Grig}
  \end{figure}
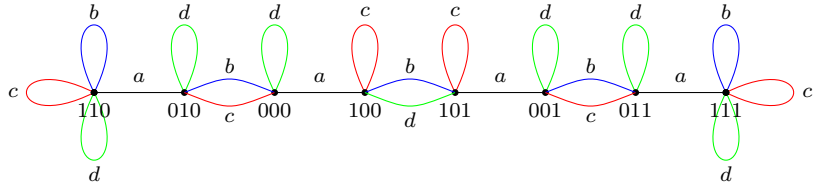

\subsection{The Fabrykowski-Gupta group}

The Fabrykowski-Gupta group is the simplest nontrivial example with $d \geq 3$. It is the spinal group $G_\omega$ defined by $d=3$, $m=1$ (so $A = \{1, a, a^2\}$ and $B = \{1, b, b^2\}$) and the constant sequence $\omega$ where all the epimorphisms equal that mapping $b$ to $a$. The graphs $\Xi$, $\Theta$ and $\Lambda = \Lambda_{\pi_b}$ are given in Fig.~\ref{fig:blocks}, and there is an example of a $\Gamma_n$ in Fig.~\ref{fig:Gamma3_FG}, which corresponds to the decription given in~\cite{BG00}.

\begin{figure}[ht]
    \centering
    \begin{tikzpicture}
	  \pgfmathsetmacro{\r}{1}
	  
	  \tikzset{
	  	pics/Xi/.style n args={3}{
	  		code = {
	  			\begin{scope}[shift={#1}]
		  			\pic {Loop={(0,0)}{60+#2}{0.5*\r}{blue}};
		  			\pic {Loop={(0,0)}{-60+#2}{0.5*\r}{red}};
		  			\draw (120 + #2: 0.35*\r) node {\tiny{$b$}};
		  			\draw (60 + #2: 0.35*\r) node {\tiny{$b^2$}};
		  			\draw[fill=black] (0,0) circle (\r/32);
		  			\draw (0+#2: 0.2*\r) node {\tiny{$#3$}};
	  			\end{scope}
	  		}	
	  	}
	  }
	  
	  \tikzset{
	  	pics/Theta/.style n args={2}{
	  		code = {
	  			\begin{scope}[shift={#1}, rotate=#2]		  			
		  			\draw (0:0) edge[midarrow>] (0: \r);
		  			\draw (0:\r) edge[midarrow>] (60: \r);
		  			\draw (60:\r) edge[midarrow>] (0:0);
		  			
		  			\draw (0.5*\r, 0.1*\r) node {\tiny{$a$}};
		  			\draw (0.7*\r, 0.37*\r) node {\tiny{$a$}};
		  			\draw (0.3*\r, 0.37*\r) node {\tiny{$a$}};
	  			\end{scope}
	  		}	
	  	}
	  }
	  
	  \tikzset{
	  	pics/Lambda/.style n args={2}{
	  		code = {
	  			\begin{scope}[shift={#1}, rotate=#2]
		  			\pic {Triangle={(0,0)}{0 + #2}{\r}{blue}};
		  			
		  			\draw[color=blue] (0:0) edge[midarrow>] (0: \r);
		  			\draw[color=blue] (0:\r) edge[midarrow>] (60: \r);
		  			\draw[color=blue] (60:\r) edge[midarrow>] (0:0);
		  			
		  			\draw (0.5*\r, 0.15*\r) node {\tiny{$b$}};
		  			\draw (0.65*\r, 0.4*\r) node {\tiny{$b$}};
		  			\draw (0.35*\r, 0.4*\r) node {\tiny{$b$}};
		  			
		  			\draw[color=red] (0:0) edge[bend left, midarrow>] (60: \r);
		  			\draw[color=red] (60:\r) edge[bend left, midarrow>] (0: \r);
		  			\draw[color=red] (0:\r) edge[bend left, midarrow>] (0:0);
		  			
		  			\draw (0.5*\r, -0.3*\r) node {\tiny{$b^2$}};
		  			\draw (\r, 0.6*\r) node {\tiny{$b^2$}};
		  			\draw (0, 0.6*\r) node {\tiny{$b^2$}};
	  			\end{scope}
	  		}	
	  	}
	  }

	  \tikzset{
	  	pics/Gamma2/.style n args={3}{
	  		code = {
	  			\begin{scope}[shift={#1}, rotate=#2]
		  			\pic {Lambda={(0,0)}{0 + #2}};
		  			
		  			\begin{scope}[shift={(0,0)}]
			  			\draw (0,-0.2*\r) node {\tiny{$02#3$}};
			  			\pic {Theta={(0+#2:0)}{180+#2}};
			  			\pic {Xi={(180+#2:\r)}{60+#2}{12#3}};
			  			\pic {Xi={(240+#2:\r)}{180+#2}{22#3}};
		  			\end{scope}
		  			
		  			\begin{scope}[shift={(60:\r)}]
			  			\draw (0,-0.2*\r) node {\tiny{$01#3$}};
			  			\pic {Theta={(0+#2:0)}{60+#2}};
			  			\pic {Xi={(60+#2:\r)}{300+#2}{11#3}};
			  			\pic {Xi={(120+#2:\r)}{60+#2}{21#3}};
		  			\end{scope}
		  			
		  			\begin{scope}[shift={(0:\r)}]
			  			\draw (0,-0.2*\r) node {\tiny{$00#3$}};
			  			\pic {Theta={(0+#2:0)}{300+#2}};
			  			\pic {Xi={(300+#2:\r)}{180+#2}{10#3}};
		  			\end{scope}
	  			\end{scope}
  			}	
  		}
	  }
	  
	  \pic {Lambda={(0,0)}{0}};
	  
	  \begin{scope}[shift={(0 : 0)}]
		  \draw (270 : 0.2*\r) node {\tiny{$201$}};
		  \begin{scope}[shift={(180 : 2*\r)}]
			  \pic {Gamma2={(0, 0)}{0}{1}};
		  \end{scope}
	  \end{scope}
	  
	  \begin{scope}[shift={(60 : \r)}]
		  \draw (150 : 0.2*\r) node {\tiny{$200$}};
		  \begin{scope}[shift={(60 : 2*\r)}]
			  \pic {Gamma2={(0, 0)}{240}{0}};
		  \end{scope}
  	  \end{scope}
  	  
  	  \begin{scope}[shift={(0 : \r)}]
	  	  \draw (30 : 0.2*\r) node {\tiny{$202$}};
	  	  \begin{scope}[shift={(300 : 2*\r)}]
	  	  \pic {Gamma2={(0, 0)}{120}{2}};
	  	  \end{scope}
  	  \end{scope}
      
	\end{tikzpicture}
  \caption{$\Gamma_3$ for the Fabrykowski-Gupta group.}
  \label{fig:Gamma3_FG}
  \end{figure}
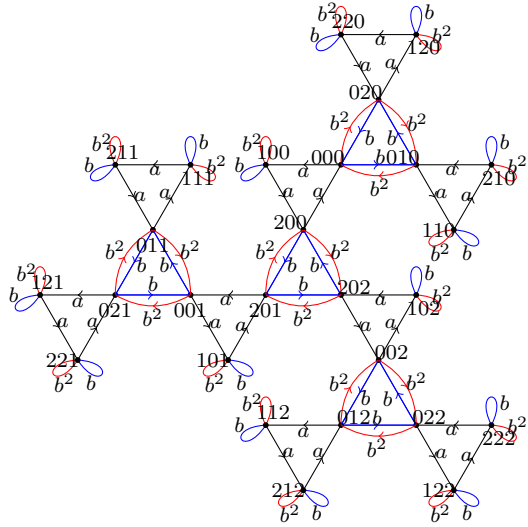
	
We provide two examples of infinite Schreier graphs, of the the points $2^\N$ and $0^\N$. The former is a one-ended graph, and the latter is two-ended. In Fig.~\ref{fig:2} and~\ref{fig:0}, vertex labels only display a prefix, so dots must be replaced with the appropriate shift of $\xi$. Any edge labeled by $\Gamma_n$ is actually a subgraph $\Delta_{\xi'}^n$, which is a copy of $\Gamma_n$.

\begin{figure}[ht]
	\centering
	\begin{tikzpicture}
		\pgfmathsetmacro{\r}{1}

		\pic {Loop={(0,0)}{90}{\r}{blue}};
		\draw (0,0) node[below] {\tiny{$\xi$}};
		\pic {Triangle={(0,0)}{0}{\r}{black}};
		\begin{scope}[shift={(60:\r)}]
			\pic {Loop={(0,0)}{0}{\r}{blue}};
			\draw (0,0) node[right] {\tiny{$1...$}};
		\end{scope}
		
		\pic {Triangle={(\r,0)}{0}{\r}{blue}};
		\draw (\r,0) node[below] {\tiny{$0...$}};
		\begin{scope}[shift={(\r,0)}]
			\begin{scope}[shift={(60:\r)}]
				\pic {Line={(0,0)}{90}{\r}{black}};
				\draw (0,0.5*\r) node[left] {\tiny{$\Gamma_1$}};
				\draw (0,0) node[right] {\tiny{$01...$}};
				\pic {Loop={(0,\r)}{0}{\r}{blue}};
				\draw (0,\r) node[right] {\tiny{$21...$}};
			\end{scope}
		\end{scope}
		\pic {Line={(2*\r,0)}{0}{\r}{black}};
		\draw (2.5*\r,0) node[above] {\tiny{$\Gamma_1$}};
		\draw (2*\r,0) node[below] {\tiny{$00...$}};
		
		\pic {Triangle={(3*\r,0)}{0}{\r}{blue}};
		\draw (3*\r,0) node[below] {\tiny{$20...$}};
		\begin{scope}[shift={(3*\r,0)}]
			\begin{scope}[shift={(60:\r)}]
				\pic {Line={(0,0)}{90}{\r}{black}};
				\draw (0,0.5*\r) node[left] {\tiny{$\Gamma_2$}};
				\draw (0,0) node[right] {\tiny{$201...$}};
				\pic {Loop={(0,\r)}{0}{\r}{blue}};
				\draw (0,\r) node[right] {\tiny{$221...$}};
			\end{scope}
		\end{scope}
		\pic {Line={(4*\r,0)}{0}{\r}{black}};
		\draw (4.5*\r,0) node[above] {\tiny{$\Gamma_2$}};
		\draw (4*\r,0) node[below] {\tiny{$200...$}};
		
		\pic {Triangle={(5*\r,0)}{0}{\r}{blue}};
		\draw (5*\r,0) node[below] {\tiny{$220...$}};
		\begin{scope}[shift={(5*\r,0)}]
		\begin{scope}[shift={(60:\r)}]
		\pic {Line={(0,0)}{90}{\r}{black}};
		\draw (0,0.5*\r) node[left] {\tiny{$\Gamma_3$}};
		\draw (0,0) node[right] {\tiny{$2201...$}};
		\pic {Loop={(0,\r)}{0}{\r}{blue}};
		\draw (0,\r) node[right] {\tiny{$2221...$}};
		\end{scope}
		\end{scope}
		\pic {Line={(6*\r,0)}{0}{\r/2}{black}};
		\draw (6.5*\r,0) node[above] {\tiny{$\Gamma_3$}};
		\draw (6*\r, 0) node[below] {\tiny{$2200...$}};
		\draw[densely dotted] (6.5*\r,0) -- (7*\r, 0);

	\end{tikzpicture}
  
	\caption{Sketch of the Schreier graph for $\xi = 2^\N$.}
  
	\label{fig:2}
\end{figure}
\begin{figure}[ht]
	\centering
	\begin{tikzpicture}
		\pgfmathsetmacro{\r}{1}

		\draw (0,0) node[below] {\tiny{$\xi$}};
		\pic {Triangle={(0,0)}{-60}{\r}{black}};
		\begin{scope}[shift={(-60:\r)}]
			\pic {Loop={(0,0)}{180}{\r}{blue}};
			\draw (0,0) node[right] {\tiny{$1...$}};
		\end{scope}
		
		\pic {Triangle={(\r,0)}{-60}{\r}{blue}};
		\draw (\r,0) node[above] {\tiny{$2...$}};
		\begin{scope}[shift={(\r,0)}]
			\begin{scope}[shift={(-60:\r)}]
				\pic {Line={(0,0)}{-90}{\r}{black}};
				\draw (0,-0.5*\r) node[left] {\tiny{$\Gamma_2$}};
				\draw (0,0) node[right] {\tiny{$201...$}};
				\pic {Loop={(0,-\r)}{180}{\r}{blue}};
				\draw (0,-\r) node[right] {\tiny{$221...$}};
			\end{scope}
		\end{scope}
		\pic {Line={(2*\r,0)}{0}{\r}{black}};
		\draw (2.5*\r,0) node[below] {\tiny{$\Gamma_2$}};
		\draw (2*\r,0) node[above] {\tiny{$202...$}};
		
		\pic {Triangle={(3*\r,0)}{-60}{\r}{blue}};
		\draw (3*\r,0) node[above] {\tiny{$222...$}};
		\begin{scope}[shift={(3*\r,0)}]
			\begin{scope}[shift={(-60:\r)}]
				\pic {Line={(0,0)}{-90}{\r}{black}};
				\draw (0,-0.5*\r) node[left] {\tiny{$\Gamma_4$}};
				\draw (0,0) node[right] {\tiny{$22201...$}};
				\pic {Loop={(0,-\r)}{180}{\r}{blue}};
				\draw (0,-\r) node[right] {\tiny{$22221...$}};
			\end{scope}
		\end{scope}
		\pic {Line={(4*\r,0)}{0}{0.5*\r}{black}};
		\draw (4.5*\r,0) node[below] {\tiny{$\Gamma_4$}};
		\draw (4*\r,0) node[above] {\tiny{$22202...$}};
		\draw[densely dotted] (4.5*\r,0) -- (5*\r, 0);

		
		\pic {Triangle={(-\r,0)}{0}{\r}{blue}};
		\begin{scope}[shift={(-\r,0)}]
			\begin{scope}[shift={(60:\r)}]
				\pic {Line={(0,0)}{90}{\r}{black}};
				\draw (0,0.5*\r) node[left] {\tiny{$\Gamma_1$}};
				\draw (0,0) node[right] {\tiny{$01...$}};
				\pic {Loop={(0,\r)}{0}{\r}{blue}};
				\draw (0,\r) node[right] {\tiny{$21...$}};
			\end{scope}
		\end{scope}
		\pic {Line={(-2*\r,0)}{0}{\r}{black}};
		\draw (-1.5*\r,0) node[above] {\tiny{$\Gamma_1$}};
		\draw (-\r,0) node[below] {\tiny{$02...$}};
		
		\pic {Triangle={(-3*\r,0)}{0}{\r}{blue}};
		\draw (-2*\r,0) node[below] {\tiny{$22...$}};
		\begin{scope}[shift={(-3*\r,0)}]
			\begin{scope}[shift={(60:\r)}]
				\pic {Line={(0,0)}{90}{\r}{black}};
				\draw (0,0.5*\r) node[left] {\tiny{$\Gamma_3$}};
				\draw (0,0) node[right] {\tiny{$2201...$}};
				\pic {Loop={(0,\r)}{0}{\r}{blue}};
				\draw (0,\r) node[right] {\tiny{$2221...$}};
			\end{scope}
		\end{scope}
		\pic {Line={(-3.5*\r,0)}{0}{0.5*\r}{black}};
		\draw (-3.5*\r,0) node[above] {\tiny{$\Gamma_3$}};
		\draw (-3*\r,0) node[below] {\tiny{$2202...$}};
		\draw[densely dotted] (-4*\r,0) -- (-3.5*\r, 0);
		
	\end{tikzpicture}
  
	\caption{Sketch of the Schreier graph for $\xi = 0^\N$.}
  
	\label{fig:0}
\end{figure}

\subsection{Grigorchuk $p$-groups}

This family was an attempt to generalize the definition of Grigorchuk $2$-groups to bigger primes $p$ in~\cite{Gri85}. We can realize this family as spinal groups by setting $d = p$ and $m = 2$, so that $B= \langle b, c \rangle$, and taking a sequence of epimorphisms in $\{ \pi_0, \dots, \pi_{d-1}, \pi\} \subset \Epi(B, A)$, where $\pi_i(b) = a$, $\pi_i(c) = a^i$ and $\pi(b) = 1$, $\pi(c) = a$. $G_\omega$ is the spinal group defined by this sequence, provided that it satisfies the kernel condition.

\subsection{\v{S}uni\'{c} groups}

Following~\cite{Sun07}, choose a prime $p$, $m \ge 1$ and a polynomial $f \in (\Z/p\Z[x])$ of degree $m$, and set $A = \Z/p\Z$ and $B = (\Z/p\Z)^m$. Then define a group $G_{p, f}$ acting on the tree $T_p$. Equivalently, one can instead of $f$ choose $p$, $m$, $\alpha \in \Epi(B, A)$ and $\rho \in \Aut(B)$ to define the group.
The group $G_{p, f}$ is in fact the spinal group defined by $d = p$,  $m$ and $\omega = \omega_0 \omega_1 \dots$ given by $\omega_i = \alpha\rho^i$, for all $i \ge 0$, whenever the action is faithful. In this case, the sequence is always periodic. All the examples mentioned above in this Section are in the class of \v{S}uni\'{c} groups.

\begin{figure}[ht]
	\centering
	\begin{tikzpicture}
		\pgfmathsetmacro{\r}{1.5}
		\pgfmathsetmacro{\t}{18}
		
		\tikzset{
			pics/Loop/.style n args={3}{
				code = {
					\begin{scope}[rotate=#1]
					\draw[color=#3] (0,0) .. controls (#2/2, #2) and (-#2/2, #2) .. (0,0);
					\draw[fill=black] (0,0) circle (#2/16);
					\end{scope}
				}	
			}
		}
		
		\foreach \i in {0,...,4}{
			\draw[color=blue] (90 + 72*\i : \r) edge[
				out = 180 + 72*\i + \t,
				in= 72 + 72*\i - \t,
				midarrow>
			] (90 + 72 + 72*\i : \r);
			
			\draw[color=red] (90 + 72*\i : \r) edge[
				out = 180 + 72 + 72*\i - \t/2,
				in= 72 + 72*\i + \t/2,
				midarrow>
			] (90 + 2*72 + 72*\i : \r);
			
			\draw[color=green] (90 + 72*\i : \r) edge[
			out = 4*72 + 72*\i - \t/2,
			in= 72 + 36 + 72*\i + \t/2,
			midarrow>
			] (90 + 3*72 + 72*\i : \r);
			
			\draw[color=orange] (90 + 72*\i : \r) edge[
			out = -72 + 72*\i + \t,
			in= 180 + 72*\i - \t,
			midarrow>
			] (90 + 4*72 + 72*\i : \r);
		}
		
		\foreach \j in {0,...,4}{
			\begin{scope}[shift={(90 + 72*\j:\r)}, shift={(54 + 72*\j:\r/2)}]
				\foreach \i in {0,...,4}{
					\draw (90 + 72*\i : \r/2) edge[midarrow>] (90 + 72 + 72*\i : \r/2);
					\draw (90 + 72*\i : \r/2) edge[midarrow>] (90 + 2*72 + 72*\i : \r/2);
					\draw (90 + 72*\i : \r/2) edge[midarrow>] (90 + 3*72 + 72*\i : \r/2);
					\draw (90 + 72*\i : \r/2) edge[midarrow>] (90 + 4*72 + 72*\i : \r/2);
					\ifthenelse{\i = \j}{}{
						\begin{scope}[shift={(90 + 2*72 + 72*\i : \r/2)}]
							\pic {Loop={2*72 + 72*\i + 2*\t}{\r/4}{blue}};
							\pic {Loop={2*72 + 72*\i + 2*\t/3}{\r/4}{red}};
							\pic {Loop={2*72 + 72*\i - 2*\t/3}{\r/4}{green}};
							\pic {Loop={2*72 + 72*\i - 2*\t}{\r/4}{orange}};
						\end{scope}
					}
				}
				\draw[fill=black] (90 + 2*72 + 72*\j: \r/2) circle (\r/64);
			\end{scope}
		}

	\end{tikzpicture}

	\label{fig:Gamma2_p5m1}  
	\caption{$\Gamma_2$ for the group defined by $p=5$, $m=1$ and a constant sequence of epimorphisms. Colored edges are labeled by powers of $b$, while black ones are labeled by powers of $a$.}
\end{figure}
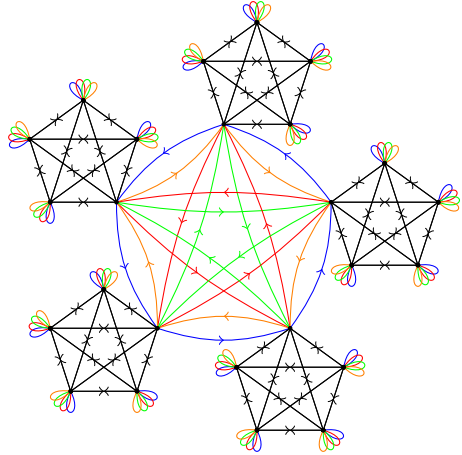



\begin{proposition}
	Let $G_\omega$ be a spinal group. Then, $G_\omega$ is self-similar iff $G_\omega$ belongs to \v{S}uni\'{c}'s family. Equivalently, iff for every $n \ge 0$, $\omega_n = \omega_0 \rho^n$, for some $\rho \in \Aut(B)$.
\end{proposition}
\begin{proof}
	It is clear from the definition that any group in \v{S}uni\'{c}'s family is self-similar. If $G_\omega$ is a self-similar spinal group, then for every $b\in B$ there must exist some $b' \in B$ such that $b' = b|_{d-1}$, so $\omega_{n+1}(b) = \omega_n(b')$ for every $n \ge 0$. If there was $b''$ for which this is also true, then $b'(b'')^{-1}$ would be in $\Ker(\omega_n)$ for every $n \ge 0$, which would violate the kernel condition. So $b'$ is unique, and we can define an endomorphism $\rho: B \to B$ as $\rho(b) = b'$. If $1 \neq b \in \Ker(\rho)$, then $\omega_n(b) = 1$ for every $n \ge 1$, which again would violate the kernel condition, so $\rho$ is an automorphism. If we define $\alpha$ to be $\omega_0$, we are in the setting of a group in \v{S}uni\'{c}'s family, with the kernel condition making the action faithful.
\end{proof}

	
	
	

\subsection{Iterated monodromy groups}

According to \cite{Nek09}, 
	an automorphism group of a rooted tree $T$ is an iterated monodromy group of a post-critically finite backward iteration of polynomials if and only if it is generated by a dendroid set of automorphisms of T.

We will omit the details, but it can be proven that any spinal group $G_\omega$ for which at most $m$ different $\Ker(\omega_i)$ occur, for all $i \ge 0$, admits a dendroid generating set.
This means that there are many iterated monodromy groups in the spinal family, self-similar or not. A particular case of iterated monodromy groups of a sequence of polynomials is when that sequence is constant equal to $f$, and so we call the iterated monodromy group $IMG(f)$. In our setting, this happens only when $m = 1$. Further details can be found in~\cite{Nek09}.

\section*{Acknowledgments}

We gratefully acknowledge support of the Swiss Natonal Science Foundation. T.N.'s work was also supported by the grant No 14.W03.31.0030 of the government of RF.

\bibliography{my.bib}{}
\bibliographystyle{plain}

\end{document}